\documentclass[12pt]{amsart}
\usepackage{amssymb,times,epsfig}

\makeatletter
\def\@strippedMR{}
\def\@scanforMR#1#2#3\endscan{
  \ifx#1M\ifx#2R\def\@strippedMR{#3}
  \else\def\@strippedMR{#1#2#3}
  \fi\fi}
\renewcommand\MR[1]{\relax\ifhmode\unskip\spacefactor3000 \space\fi
  \@scanforMR#1\endscan
  MR\MRhref{\@strippedMR}{\@strippedMR}}
\makeatother

\newtheorem*{Thm*}{Theorem}
\newtheorem{Thm}{Theorem}
\newtheorem{Cor}[Thm]{Corollary}
\newtheorem{Prop}[Thm]{Proposition}
\newtheorem{Lemma}[Thm]{Lemma}

\theoremstyle{definition}
\newtheorem{Defn}{Definition}

\newtheorem{Remark}[Defn]{Remark}
\newtheorem{Example}[Defn]{Example}
\newtheorem{Defn-Remark}[Defn]{Definition and Remark}

\newcommand{\mf}[1]{\mathbb{#1}}
\newcommand{\mc}[1]{\mathcal{#1}}
\newcommand{\mb}[1]{\mathbf{#1}}

\DeclareMathOperator{\NC}{\mathit{NC}}

\DeclareMathOperator{\Part}{\mathcal{P}}
\DeclareMathOperator{\Int}{\mathit{Int}}
\DeclareMathOperator{\Inner}{\mathit{Inner}}
\DeclareMathOperator{\Outer}{\mathit{Outer}}
\DeclareMathOperator{\Sing}{\mathit{Sing}}
\DeclareMathOperator{\SC}{\mathit{SC}}

\newcommand{\norm}[1]{\left\Vert#1\right\Vert}
\newcommand{\abs}[1]{\left\vert#1\right\vert}
\newcommand{\chf}[1]{\mathbf{1}_{#1}}
\newcommand{\set}[1]{\left\{#1\right\}}
\newcommand{\ip}[2]{\left \langle #1, #2 \right \rangle}
\newcommand{\state}[1]{\varphi \left[ #1 \right]}
\newcommand{\State}[1]{\Phi \left[ #1 \right]}
\renewcommand{\phi}{\varphi}

\newcommand{\Cum}[2]{R^{#1} \left[ #2 \right]}
\newcommand{\CumFun}[2]{R^{#1} \left( #2 \right)}

\newcommand{\A}[2]{A^{#1} \left( #2 \right)}
\newcommand{\Aa}[2]{A_{#1} \left( #2 \right)}
\newcommand{\Var}[1]{\mathrm{Var} \left[ #1 \right]}

\newcommand{\br}{\medskip\noindent}

\allowdisplaybreaks[1]

\title[Appell polynomials III. C-free theory.]{Appell polynomials and their relatives III. Conditionally free theory.}
\author[M.~Anshelevich]{Michael Anshelevich}
\thanks{This work was supported in part by NSF grant DMS-0613195}
\address{Department of Mathematics, Texas A\&M University, College Station, TX 77843-3368}
\email{manshel@math.tamu.edu}
\date{\today}

\begin{document}

\begin{abstract}
We extend to the multivariate non-commutative context the descriptions of a ``once-stripped'' probability measure in terms of Jacobi parameters, orthogonal polynomials, and the moment generating function. The corresponding map $\Phi$ on states was introduced previously by Belinschi and Nica. We then relate these constructions to the c-free probability theory, which is a version of free probability for algebras with two states, introduced by Bo{\.z}ejko, Leinert, and Speicher. This theory includes as two extreme cases the free and Boolean probability theories. The main objects in the paper are the analogs of the Appell polynomial families in the two state context. They arise as fixed points of the transformation which takes a polynomial family to the associated polynomial family (in several variables), and their orthogonality is also related to the map $\Phi$ above. In addition, we prove recursions, generating functions, and factorization and martingale properties for these polynomials, and describe the c-free version of the Kailath-Segall polynomials, their combinatorics, and Hilbert space representations.
\end{abstract}

\maketitle

\section{Introduction}

\noindent
The title of the paper will be explained below, but we start with some very classical results.

\br
Let $\mu$ be a probability measure on the real line, all of whose moments are finite, which we normalize to have mean zero and variance one. It has a sequence of monic orthogonal polynomials $\set{P_n}$, which satisfy a three term recursion relation. The coefficients in this recursion are the Jacobi parameters of the measure. The measure also has a family of orthogonal polynomials of the second kind $\set{Q_n}$, which satisfy the same recursion with different initial conditions. They are orthogonal with respect to the measure $\nu$ which is the ``once-stripped'' \cite{Damanik-Simon-periodic} version of $\mu$ : the Jacobi parameter sequences of $\nu$ are obtained by removing the first terms of the Jacobi sequences of $\mu$. It is a classical fact going back to Darboux that equivalently,
\[
Q_{n-1}(x) = \int_{\mf{R}} \frac{P_n(x) - P_n(y)}{x-y} \,d\mu(y).
\]
Using continued fraction expansions for the moment generating functions $M^\mu$ and $M^\nu$ of $\mu$ and $\nu$, one obtains the third equivalent description:
\[
1 - M^\mu(z)^{-1} = z^2 M^\nu(z).
\]

\br
The paper starts with an extension of the equivalence between these three results to the multivariate context of \emph{non-commutative} polynomials, with measures being replaced by states (from a special but large class described in \cite{AnsMonic} of states which have monic orthogonal polynomials). More precisely, it turns out that in the multivariate context, the map $\mu \mapsto \nu$ is in general not defined, and we consider its inverse $\Phi: \nu \mapsto \mu$. The map $\Phi$ was introduced by Belinschi and Nica \cite{Belinschi-Nica-B_t,Belinschi-Nica-Free-BM}.

\br
It is easy to see that the unique fixed point of $\Phi$ is the semicircular distribution, and correspondingly the Chebyshev polynomials of the second kind are the unique orthogonal polynomials which satisfy
\[
U_{n-1}(x) = \int_{\mf{R}} \frac{U_n(x) - U_n(y)}{x-y} \,d\mu(y).
\]
However, if we remove the orthogonality requirement and only ask that the polynomials be centered with respect to $\mu$, then we show that such a family exists for any $\mu$. In fact, these are the free Appell polynomials, which were defined in the first paper \cite{AnsAppell} of this series using the following recursion involving the difference quotient:
\[
\frac{A_n(x) - A_n(y)}{x-y} = \sum_{k=0}^{n-1} A_k(x) A_{n-k-1}(y)
\]
(these objects were also considered earlier in \cite{Verde-Star}). The parallel here is with Paul Appell's \cite{Appell} differential recursion
\[
A_n'(x) = n A_{n-1}(x).
\]
The free Appell polynomials have a number of properties which resemble those of the usual Appell families. Moreover, they turned out to be related to free probability \cite{VDN,Nica-Speicher-book}. In the second paper \cite{AnsBoolean}, we performed similar analysis for polynomial families related to Boolean probability theory, one of only two other natural non-commutative probability theories in addition to the free one (and of course the usual theory). In fact, even though the free and Boolean theories are quite different, the corresponding polynomial families turned out to be closely related, and in particular their Meixner families coincide.

\br
The initial motivation for this paper was to explain this fact. Both the free and the Boolean setting are in fact particular cases of a more general construction for a space with \emph{two} expectations, or more precisely an algebra with two states $(\mc{A}, \phi, \psi)$. We define the Appell polynomial families in this setting, so that they restrict correctly to the two cases above. All the familiar results---recursion relations, generating functions, relations to partition lattices and cumulants, Kailath-Segall expansions, and martingale properties---hold in this case, none of which will come as a surprise to the readers familiar with the first two papers in the series. The main function of these results is to confirm that the definition of the Appell polynomials in the two-state setting is the correct one, and unify the free and Boolean constructions. In one variable, the c-free Appell polynomials have also appeared in \cite{Verde-Star} under the name of associated sequences.

\br
More interestingly, the two state free probability theory is directly relevant to the discussion at the beginning of the introduction, which involves two measures $\mu, \nu$. Namely, to the three equivalent definitions of the relation $\mu = \Phi[\nu]$ we now add three more. First, the two state free cumulant generating function for the pair $(\mu, \nu)$ is in this case simple quadratic. Recall that the unique fixed point of $\Phi$ is the semicircle law, and that in the usual free probability theory, only the second free cumulant of this law is non-zero. Second, if $\mu = \Phi[\nu]$, the c-free Appell polynomials are almost orthogonal. They are orthogonal if in addition, $\nu$ is a free product of semicircular distributions, in which case $\mu$ is a free Meixner state. The role of the free Meixner states, and of the map $\Phi$ and its generalizations, in free probability theory with two states is investigated in more detail in \cite{AnsEvolution}. Third, such a pair $(\mu, \nu)$ minimizes a certain type of free Fisher information.

\br
In the free and Boolean theory, the free Meixner distributions arose as those distributions whose orthogonal polynomials are also generalized Appell (more precisely, Sheffer). As mentioned above, in the two-state theory these distributions arise much more naturally, namely they are the ones for which the Appell polynomials themselves are orthogonal. In retrospect, this statement should have been expected. One of the ways to describe the free Meixner distributions is that their Jacobi parameter sequences are constant after the first step. In the language of \cite{AccBozGaussianization}, this corresponds to looking at partitions for which one distinguishes the classes at depth one from the other classes, and this is exactly the underlying combinatorics of the two-state theory. It is then natural to ask for the relation between distributions whose Jacobi parameter sequences are constant after some point (considered for example in \cite{Kato-Mixed-periodic-continued-fractions}) and ``$n$-state'' probability theories. Such theories have indeed been attempted \cite{CD-Ionescu,Mlo99}, but the resulting products are not associative.

\br
The paper is organized as follows. After a brief preliminary Section~\ref{Section:Preliminaries}, in Section~\ref{Section:Second-kind} we prove the equivalence between various descriptions of multivariate non-commutative orthogonal polynomials of the second kind. The next section provides background on non-commutative probability theories. Section~\ref{Section:Appell} treats the c-free Appell polynomials, their orthogonality, and further properties of the transformation $\Phi$. And Section~\ref{Section:Processes} finishes with martingale properties and Fock space representations for the c-free Appell polynomials, which parallel the results in \cite{AnsAppell,AnsBoolean}.

\br
\textbf{Acknowledgements.}
I would like to thank Andu Nica and Serban Belinschi for a number of discussions which contributed to the development of this paper, and for explaining their work to me. I also thank Jamie Mingo for a useful discussion.

\section{Preliminaries I}
\label{Section:Preliminaries}

\noindent
We will freely use the notions and notation from the Preliminaries section of \cite{AnsBoolean}; here we list the highlights.

\subsection{Polynomials and power series}
Let $\mf{C}\langle \mb{x} \rangle = \mf{C}\langle x_1, x_2, \ldots, x_d \rangle$ be all the polynomials with complex coefficients in $d$ non-com\-mu\-ting variables. They form a unital $\ast$-algebra.

\br
For $i = 1, \ldots, d$, define the partial difference quotient operator
\[
\partial_i: \mf{C} \langle \mb{x} \rangle \rightarrow \mf{C} \langle \mb{x} \rangle \otimes \mf{C} \langle \mb{x} \rangle
\]
by a linear extension of $\partial_i(1) = 0$,
\[
\partial_i (x_{u(1)} x_{u(2)} \ldots x_{u(n)}) = \sum_{j: u(j) = i} x_{u(1)} \ldots x_{u(j-1)} \otimes x_{u(j+1)} \ldots x_{u(n)}.
\]
Also, for a non-commutative power series $G$ in
\[
\mb{z} = (z_1, z_2, \ldots, z_d),
\]
define the left non-commutative partial derivative $D_i G$ by a linear extension of $D_i(1) = 0$,
\[
D_i z_{\vec{u}} = \delta_{i u(1)} z_{u(2)} \ldots z_{u(n)}.
\]

\br
A monic polynomial family in
\[
\mb{x} = (x_1, x_2, \ldots, x_d)
\]
is a family $\set{P_{\vec{u}}(\mb{x})}$ indexed by all multi-indices
\[
\bigcup_{k=1}^\infty \set{\vec{u} \in \set{1, \ldots, d}^k}
\]
(with $P_{\emptyset} = 1$ being understood) such that\[
P_{\vec{u}}(\mb{x}) = x_{\vec{u}} + \textsl{lower-order terms}.
\]

\subsection{Algebras and states}
Algebras $\mc{A}$ in this paper will always be complex $\ast$-algebras and, unless stated otherwise, unital. If the algebra is non-unital, one can always form its unitization $\mf{C} 1 \oplus \mc{A}$; if $\mc{A}$ was a $C^\ast$-algebra, its unitization can be made into one as well.

\br
Functionals $\mc{A} \rightarrow \mf{C}$ will always be linear, unital, and $\ast$-compatible. A state is a functional which in addition is positive definite, that is
\[
\state{X^\ast X} \geq 0
\]
(zero value for non-zero $X$ is allowed).

\br
Most of the time we will be working with states on $\mf{C} \langle \mb{x} \rangle$ arising as joint distributions. For
\[
X_1, X_2, \ldots, X_d \in \mc{A}^{sa},
\]
their joint distribution with respect to $\psi$ is a state on $\mf{C} \langle \mb{x} \rangle$ determined by
\[
\state{P(\mb{x})} = \psi^{X_1, X_2, \ldots, X_d}\left[P(x_1, x_2, \ldots, x_d)\right] = \psi \left[P(X_1, X_2, \ldots, X_d)\right].
\]
The numbers $\state{x_{\vec{u}}}$ are the moments of $\phi$. More generally, for $d$ non-commuting indeterminates $\mb{z} = (z_1, \ldots, z_d)$, the series
\[
M(\mb{z}) = \sum_{\vec{u}} \state{x_{\vec{u}}} z_{\vec{u}}
\]
is the moment generating function of $\phi$.

\br
A state $\phi$ on $\mf{C} \langle \mb{x} \rangle$ has a monic orthogonal polynomial system, or MOPS, if for any multi-index $\vec{u}$, there is a monic polynomial $P_{\vec{u}}$ with leading term $x_{\vec{u}}$, such that these polynomials are orthogonal with respect to $\phi$, that is,
\[
\ip{P_{\vec{u}}}{P_{\vec{v}}}_\phi = 0
\]
for $\vec{u} \neq \vec{v}$.

\br
For a probability measure $\mu$ on $\mf{R}$ all of whose moments are finite, its monic orthogonal polynomials $\set{P_n}$ satisfy three-term recursion relations
\begin{equation}
\label{Three-term-recursion}
x P_n(x) = P_{n+1}(x) + \beta_n P_n(x) + \gamma_n P_{n-1}(x),
\end{equation}
with initial conditions $P_{-1} = 0$, $P_0 = 1$. We will call the parameter sequences
\[
(\beta_0, \beta_1, \beta_2, \ldots), (\gamma_1, \gamma_2, \gamma_3, \ldots)
\]
the Jacobi parameter sequences for $\mu$. Generalizations of such parameters for states with MOPS were found in \cite{AnsMulti-Sheffer}, so that every such state is of the form $\phi_{\set{\mc{T}_i}, \mc{C}}$ for two families of matrices $\set{\mc{T}_i^{(n)}}$, $\set{\mc{C}^{(n)}}$.

\subsection{Free Appell polynomials}
\label{Subsection:Free-Appell}
Let $(\mc{A}, \psi)$ be an algebra with a linear functional. The free Appell polynomials are, for each $n \in \mf{N}$, maps
\[
A^{\psi}: \left(\mc{A}^{sa}\right)^n \rightarrow \mc{A}, \qquad (X_1, X_2, \ldots, X_n) \mapsto \A{\psi}{X_1, X_2, \ldots, X_n}
\]
such that $\A{\psi}{X_1, X_2, \ldots, X_n}$ is a polynomial in $X_1, X_2, \ldots, X_n$,
\begin{equation}
\label{Free-Appell-definition-1}
\partial_i \A{\psi}{X_1, X_2, \ldots, X_n} = \A{\psi}{X_1, \ldots, X_{i-1}} \otimes \A{\psi}{X_{i+1}. \ldots, X_n},
\end{equation}
with the obvious modifications for $i = 1, n$ corresponding to $\A{\psi}{\emptyset} = 1$, and
\begin{equation}
\label{Free-Appell-definition-2}
\psi \left[\A{\psi}{X_1, X_2, \ldots, X_n}\right] = 0
\end{equation}
for $n \geq 1$. They were introduced in \cite{AnsAppell} via a slightly more complicated but equivalent definition.

\section{Non-commutative orthogonal polynomials of the second kind}
\label{Section:Second-kind}

\begin{Remark}[Polynomials of the first and second kind]
\label{Remark:Polynomials-first-kind}
For a measure $\mu$ all of whose moments are finite, its monic orthogonal polynomials $\set{P_n}$ satisfy a three-term recursion relation \eqref{Three-term-recursion}, with initial conditions $P_{-1} = 0$, $P_0 = 1$. Since this recursion is second order, there is another family of polynomials satisfying the same recursion, with initial conditions $Q_0 = 0$, $Q_1 = 1$, which can also be defined via
\[
Q_{n-1}(x) = (I \otimes \mu) [\partial P_n].
\]
These are polynomials of the first and second kind corresponding to $\mu$, see Section 1.2.1 of \cite{Akh65}; one also says that $\set{Q_n}$ are associated to $\set{P_n}$. In general, $\set{Q_n}$ are orthogonal with respect to a different measure, obtained from $\mu$ by deleting the first terms of its Jacobi parameter sequences. There is a unique case when the formula above does not give a new family: the Chebyshev polynomials of the second kind are a fixed point for this operation (up to a shift in degree), and are associated to themselves. We now show that, if we drop the condition of orthogonality and only keep centeredness, i.e.\ orthogonality to the constants, such fixed points are exactly the free Appell polynomials. The following proposition actually describes a more general multivariate case.
\end{Remark}

\begin{Prop}
\label{Prop:Appell-characterization}
For any $\psi$, the free Appell polynomials are the unique polynomial family satisfying
\[
(I \otimes \psi) \partial_i P\left(X_1, X_2, \ldots, X_n\right)
= \delta_{i n} P\left(X_1, \ldots, X_{n-1} \right)
\]
and
\[
\psi \left[P\left(X_1, X_2, \ldots, X_n \right) \right] = 0.
\]
\end{Prop}

\begin{proof}
It immediately follows from their definition that the free Appell polynomials satisfy these properties. To prove uniqueness, it suffices to show that the map
\[
P \mapsto \Bigl((I \otimes \psi) \partial_1 P, (I \otimes \psi) \partial_2 P, \ldots, (I \otimes \psi) \partial_d P \Bigr)
\]
on polynomials contains only constants in its kernel. For this, in turn, it suffices to show that the images under this map of different monomials $x_{\vec{u}}$ are linearly independent. Indeed, take any distinct $\vec{u}_1, \vec{u}_2, \ldots, \vec{u}_k$. Choose some $i, j$ with $(I \otimes \psi) \partial_i (x_{\vec{u}_j})$ of the highest degree. Then, denoting $n = \abs{\vec{u}_j}$, it follows that $u_j(n) = i$ and all $\abs{\vec{u}_s} \leq n$. So $(I \otimes \psi) \partial_i (x_{\vec{u}_j})$ contains the term $x_{u_j(1)} x_{u_j(2)} \ldots x_{u_j(n-1)}$, and the only way one of the other $(I \otimes \psi) \partial_i (x_{\vec{u}_s})$ could contain this term is if $\vec{u}_s = \vec{u}_j$, which is not the case.
\end{proof}

\begin{Defn-Remark}
In \cite{Belinschi-Nica-Free-BM}, Belinschi and Nica defined a map $\Phi$ from states to states via
\[
\eta^{\State{\psi}}(\mb{w}) = 1 - (1 + M^{\State{\psi}}(\mb{w}))^{-1} = \sum_{i=1}^d w_i (1 + M^\psi(\mb{w})) w_i.
\]
As mentioned in the introduction, in one variable this map can be described as the transformation coming from the shift on Jacobi parameter sequences: if $\mu$ has the Jacobi parameter sequences
\[
\set{(\beta_0, \beta_1, \beta_2, \ldots), (\gamma_1, \gamma_2, \gamma_3, \ldots)},
\]
it follows from the continued fraction representations for $\eta^{\State{\mu}}(w)$ and $M^\mu(w)$ that the Jacobi parameter sequences for $\State{\mu}$ are
\[
\set{(0, \beta_0, \beta_1, \ldots), (1, \gamma_1, \gamma_2, \ldots)}.
\]
In particular, the free Meixner distributions (Section~\ref{Subsection:Free-Meixner}) are exactly the images under $\Phi$ of various semicircular distributions:
\[
\mu_{b, c} = \State{\SC(b, 1+c)}.
\]
\end{Defn-Remark}

\begin{Thm}
\label{Thm:OPS-second-kind}
Let $\psi$ be a state with MOPS $\set{Q_{\vec{u}}(\mb{x})}$, corresponding to
\[
\set{\mc{T}_i^{(n)}, 1 \leq i \leq d, n \geq 0}, \set{\mc{C}^{(n)}, n \geq 1},
\]
where $\mc{T}_i^{(n)}, \mc{C}^{(n)}$ are $d^n \times d^n$ matrices. The following are equivalent conditions on a state $\phi$ with mean zero and identity covariance.
\begin{enumerate}
\item
$\phi = \State{\psi}$, that is, $1 - (1 + M^\phi(\mb{w}))^{-1} = \sum_{i=1}^d w_i (1 + M^\psi(\mb{w})) w_i$.
\item
$\phi$ is the state with MOPS corresponding to the matricial sequences
\[
\set{0, \mc{T}_i^{(n-1)} \otimes I, 1 \leq i \leq d, n \geq 1}, \set{I, \mc{C}^{(n-1)} \otimes I, n \geq 2}.
\]
\item
$\phi$ is a state with MOPS $\set{P_{\vec{u}}(\mb{x})}$ which satisfy
\[
(I \otimes \phi) \partial_j P_{(\vec{u}, m)}(\mb{x}) = 0
\]
for $j \neq m$ and
\[
Q_{\vec{u}}(\mb{x}) = (I \otimes \phi) \partial_m P_{(\vec{u},m)}(\mb{x}).
\]
In this case we call the polynomials $\set{Q_{\vec{u}}(\mb{x})}$ the orthogonal polynomials of the second kind for $\phi$.
\end{enumerate}
\end{Thm}

\begin{proof}
The first two statements are equivalent using the results in the appendix of \cite{AnsBoolean}. Next we show that (c) $\Rightarrow$ (b). Let $\phi = \phi_{\set{\Delta_i}, \Gamma}$, so that by Theorem~2 of \cite{AnsMonic}, its MOPS satisfy the recursions
\[
x_i P_{(\vec{u}, m)}(\mb{x}) = P_{(i, \vec{u}, m)} + \sum_{\vec{v}, k} \Delta_{i, (\vec{v}, k), (\vec{u}, m)} P_{(\vec{v}, k)}(\mb{x}) + \delta_{i, u(1)} \Gamma_{(\vec{u}, m)} P_{(u(2), \ldots, m)}(\mb{x}).
\]
If
\[
(I \otimes \phi) \partial_j P_{(\vec{u},m)}(\mb{x}) = \delta_{jm} Q_{\vec{u}}(\mb{x}),
\]
then using
\[
(I \otimes \phi) \partial_j \Bigl( x_i P_{(\vec{u},m)}(\mb{x}) \Bigr)
= \delta_{ij} \state{P_{(\vec{u},m)}(\mb{x})} + x_i (I \otimes \phi) \partial_j P_{(\vec{u},m)}(\mb{x})
= x_i (I \otimes \phi) \partial_j P_{(\vec{u},m)}(\mb{x}),
\]
we get
\[
\begin{split}
x_i \delta_{jm} Q_{\vec{u}}(\mb{x})
& = \delta_{jm} Q_{(i, \vec{u})} + \sum_{\vec{v}, k} \delta_{jk} \Delta_{i, (\vec{v}, k), (\vec{u}, m)} Q_{\vec{v}}(\mb{x}) + \delta_{i, u(1)} \delta_{jm} \Gamma_{(\vec{u}, m)} Q_{(u(2), \ldots)}(\mb{x}) \\
& = \delta_{jm} Q_{(i, \vec{u})} + \sum_{\vec{v}} \Delta_{i, (\vec{v}, j), (\vec{u}, m)} Q_{\vec{v}}(\mb{x}) + \delta_{i, u(1)} \delta_{jm} \Gamma_{(\vec{u}, m)} Q_{(u(2), \ldots)}(\mb{x}).
\end{split}
\]
Since this equality holds for any $j = m$ and the coefficients on the right hand side are uniquely determined, it follows that they do not depend on $m$, so that $\Delta_{i, (\vec{v}, j), (\vec{u}, m)} = \delta_{jm} \mc{T}_{i, \vec{v}, \vec{u}}$ and $\Gamma_{(\vec{u}, m)} = \mc{C}_{\vec{u}}$, in other words $\Delta_i^{(n)} = \mc{T}_i^{(n-1)} \otimes I$ and $\Gamma^{(n)} = \mc{C}^{(n-1)} \otimes I$ for
\[
\psi = \phi_{\set{\mc{T}_i}, \mc{C}}.
\]
The converse is similar, and follows by induction.
\end{proof}

\begin{Remark}
Further statements equivalent to the conditions of the preceding theorem are proven in Section~\ref{Subsection:Orthogonal}:
\begin{itemize}
\item[(d)]
The two-state free cumulant generating function $R^{\phi, \psi}$ of $(\phi, \psi)$ is simple quadratic.
\item[(e)]
The c-free Appell polynomials $A^{\phi, \psi}$ are orthogonal to the degree one c-free Appell polynomials, and $\phi$ has mean zero and identity covariance (Lemma~\ref{Lemma:Quadratic}).
\item[(f)]
For fixed $\psi$, the c-free Fisher information for the pair $(\phi, \psi)$ is minimal among all states $\phi$ with mean zero and identity covariance (Proposition~\ref{Prop:Fisher}).
\end{itemize}
\end{Remark}

\section{Preliminaries II}
\label{Section:Preliminaries2}

\subsection{Partitions}
$\NC(n)$ is the lattice of non-crossing partitions of $n$ elements, and $\Int(n)$ is the corresponding lattice of interval partitions. A class $B \in \pi$ of a non-crossing partition is inner if for $j \in B$,
\[
\exists \ i \stackrel{\pi}{\sim} k \stackrel{\pi}{\not \sim} j: i < j < k,
\]
otherwise $B$ is outer. The collection of all the inner classes of $\pi$ will be denoted $\Inner(\pi)$, and similarly for $\Outer(\pi)$.

\subsection{Cumulants}
For a state $\phi$, its Boolean cumulant generating function is defined by
\[
\eta^\phi(\mb{z}) = 1 - (1 + M^\phi(\mb{z}))^{-1},
\]
and its coefficients are the Boolean cumulants of $\phi$. They can also be expressed in terms of moments of $\phi$ using the lattice of interval partitions. Similarly, the free cumulant generating function of a state $\psi$ is defined by either of the equivalent implicit equations
\begin{equation}
\label{R-w-M}
M^\psi(\mb{w}) = \CumFun{\psi}{\mb{w} (1 + M^\psi(\mb{w}))}
\end{equation}
or
\begin{equation}
\label{R-M-w}
M^\psi(\mb{w}) = \CumFun{\psi}{(1 + M^\psi(\mb{w})) \mb{w}},
\end{equation}
where
\[
(1 + M^\psi(\mb{w})) \mb{w} = \Bigl( (1 + M^\psi(\mb{w})) w_1, (1 + M^\psi(\mb{w})) w_2, \ldots, (1 + M^\psi(\mb{w})) w_d \Bigr).
\]
We will frequently, sometimes without comment, use the change of variables
\begin{equation}
\label{z-M-w}
z_i = (1 + M^\psi(\mb{w})) w_i , \qquad w_i = (1 + \CumFun{\psi}{\mb{z}})^{-1} z_i.
\end{equation}
The coefficients of $\CumFun{\psi}{\mb{z}}$ are the free cumulants of $\psi$, and can also be expressed in terms of the moments of $\psi$ using the lattice of non-crossing partitions.

\subsection{Two-state cumulants}
The typical setting in this paper is a triple $(\mc{A}, \phi, \psi)$, where $\mc{A}$ is an algebra and $\phi, \psi$ are functionals on it. If necessary, we will assume that both $\phi, \psi$ are states with MOPS. By rotation, we can assume without loss of generality that $\phi$ is normalized to have zero means and identity covariance. Then no such assumptions can be made on $\psi$, but the MOPS condition still guarantees that the covariance of $\psi$ is diagonal.

\br
We define the two-state free cumulants of the pair $(\phi, \psi)$ via
\[
\state{x_1 \ldots x_n} = \sum_{\pi \in \NC(n)} \prod_{B \in \Outer(\pi)} \Cum{\phi, \psi}{\prod_{i \in B} x_i} \prod_{C \in \Inner(\pi)} \Cum{\psi}{\prod_{j \in C} x_j}.
\]
Their generating function is
\[
\CumFun{\phi, \psi}{\mb{z}} = \sum_{\vec{u}} \Cum{\phi, \psi}{x_{\vec{u}}} z_{\vec{u}}.
\]
Equivalently (up to changes of variables, this is Theorem~5.1 of \cite{BLS96}), we could have defined the two-state free cumulant generating function via the condition
\begin{equation}
\label{eta-M-R}
\eta^\phi(\mb{w}) = (1 + M^\psi(\mb{w}))^{-1} \CumFun{\phi, \psi}{(1 + M^\psi(\mb{w})) \mb{w}}.
\end{equation}
We also note that for $z_i = (1 + M^\psi(\mb{w})) w_i$,
\begin{equation}
\label{C-cumulant-identity}
\begin{split}
1 + \CumFun{\psi}{\mb{z}} - \CumFun{\phi, \psi}{\mb{z}}
& = 1 + M^\psi(\mb{w}) - \CumFun{\phi, \psi}{(1 + M^\psi(\mb{w})) \mb{w}} \\
& = (1 + M^\psi(\mb{w})) (1 + M^\phi(\mb{w}))^{-1}.
\end{split}
\end{equation}

\br
For elements $X_1, X_2, \ldots, X_n \in \mc{A}^{sa}$, we will denote their joint cumulants
\[
\Cum{\phi, \psi}{X_1, X_2, \ldots, X_n} = \Cum{\phi^{X_1, X_2, \ldots, X_n}, \psi^{X_1, X_2, \ldots, X_n}}{x_1, x_2, \ldots, x_n}
\]
to be the corresponding joint cumulants with respect to their joint distributions.

\begin{Defn}
\label{Defn:c-free}
Let $(\mc{A}, \phi, \psi)$ be an algebra with two states.
\begin{enumerate}
\item
Subalgebras $\mc{A}_1, \ldots, \mc{A}_d \subset \mc{A}$ are conditionally free, or c-free, with respect to $(\phi, \psi)$ if for any $n \geq 2$,
\[
X_i \in \mc{A}_{u(i)}, \quad i = 1, 2, \ldots, n, \qquad u(1) \neq u(2) \neq \ldots \neq u(n),
\]
the relation
\[
\psi[X_1] = \psi[X_2] = \ldots = \psi[X_n] = 0
\]
implies
\begin{equation}
\label{centered-product}
\state{X_1 X_2 \ldots X_n} = \state{X_1} \state{X_2} \ldots \state{X_n}.
\end{equation}
\item
The subalgebras are $(\phi | \psi)$ free if for $X_1, X_2, \ldots, X_n \in \bigcup_{j=1}^d \mc{A}_j$,
\[
\Cum{\phi, \psi}{X_1, X_2, \ldots, X_n} = 0
\]
unless all $X_i$ lie in the same subalgebra.
\end{enumerate}
\end{Defn}

\noindent
As pointed out in \cite{Boz-Bryc-Two-states}, these properties are \emph{not} equivalent, however they become equivalent under the extra requirement that the subalgebras are $\psi$-freely independent. In any case, throughout most of the paper we will be working with cumulants, and will only invoke conditional freeness itself in Section~\ref{Subsection:Processes}.

\begin{Example}
If $X, Z$ are c-free from $Y$, then (Lemma 2.1 of \cite{BLS96})
\[
\state{X Y} = \state{X} \state{Y},
\]
\[
\state{X Y Z}
= \state{X} \state{Y} \state{Z} + \left( \state{X Z} - \state{X} \state{Z} \right) \psi[Y].
\]
\end{Example}

\begin{Lemma}
\label{Lemma:Endpoints-c-free}
Under the hypothesis of the preceding definition, the conclusion \eqref{centered-product} remains valid without any assumptions on the endpoints $\psi[X_1]$ and $\psi[X_n]$
\end{Lemma}

\begin{proof}
For $n=2$, the result is stated in the preceding example. For $n \geq 3$, denote $Y = X_2 \ldots X_{n-1}$. Then
\[
\begin{split}
\state{X_1 Y X_n}
& = \state{(X_1 - \psi[X_1]) Y (X_n - \psi[X_n])} + \psi[X_1] \state{Y (X_n - \psi[X_n])} \\
&\quad + \state{(X_1 - \psi[X_1]) Y} \psi[X_n] + \psi[X_1] \state{Y} \psi[X_n] = 0,
\end{split}
\]
since for each of these terms, the argument of $\phi$ satisfies the hypothesis of the definition.
\end{proof}

\begin{Example}
\label{Example:Conditional-freeness}
The following are important particular cases of conditional freeness.
\begin{enumerate}
\item
If $\phi = \psi$, so that $(\mc{A}, \phi)$ is an algebra with a single state, conditional freeness with respect to $(\phi, \phi)$ is the same as free independence with respect to $\phi$. Moreover, $R^{\phi, \phi} = R^\phi$.
\item
If $\mc{A}$ is a non-unital algebra, define a state $\delta_0$ on its unitization $\mf{C} 1 \oplus \mc{A}$ by $\delta_0[1] = 1$, $\delta_0[\mc{A}] = 0$. Then conditional freeness of subalgebras $(\mf{C} 1 \oplus \mc{A}_1), \ldots, (\mf{C} 1 \oplus \mc{A}_d)$ with respect to $(\phi, \delta_0)$ is the same as Boolean independence of subalgebras $\mc{A}_1, \ldots, \mc{A}_d$ with respect to $\phi$. Moreover, $R^{\phi, \delta_0} = \eta^\phi$. The Boolean theory has been treated as a particular case of the c-free theory in \cite{Franz-Boolean-circle} and in a number of other sources.
\item
Specializing the preceding example, if $\mc{A} = \mf{C} \langle \mb{x} \rangle$, it is a unitization of the algebra of polynomials without constant term, and $\delta_0[P]$ is the constant term of a polynomial, so that we denote, even for non-commuting polynomials,
\[
\delta_0[P] = P(0).
\]
\end{enumerate}
\end{Example}

\noindent
See \cite{Lenczewski-Unification,Mlotkowski-Operator-conditional,Yoshida-Delta,Lehner-Cumulants-I,Oravecz-Pure-convolutions,Franz-Multiplicative-monotone,Popa-c-free-amalgamation}, as well as references in \cite{AnsEvolution} for other particular cases and generalizations of conditional freeness; the appearance of the free Meixner laws (see below) in related contexts has been observed even more widely.

\subsection{Convolutions}
\label{Subsection:Convolutions}
If $\phi, \psi$ are two unital linear functionals on $\mf{C} \langle \mb{x} \rangle$, then $\phi \boxplus \psi$ is their free convolution, that is a unital linear functional on $\mf{C} \langle \mb{x} \rangle$ determined by
\[
\CumFun{\phi}{\mb{z}} + \CumFun{\psi}{\mb{z}} = \CumFun{\phi \boxplus \psi}{\mb{z}}.
\]
Similarly, $\phi \uplus \psi$, their Boolean convolution, is a unital linear functional on $\mf{C} \langle \mb{x} \rangle$ determined by
\[
\eta^\phi(\mb{z}) + \eta^\psi(\mb{z}) = \eta^{\phi \uplus \psi}(\mb{z}).
\]
See Lecture~12 of \cite{Nica-Speicher-book} for the relation between free convolution and free independence; the relation in the Boolean case is similar.

\subsection{Free Meixner distributions and states}
\label{Subsection:Free-Meixner}
The semicircular distribution with mean $\alpha$ and variance $\beta$ is
\[
d\SC(\alpha, \beta)(x) = \frac{1}{2 \pi \beta} \sqrt{4 \beta - (x - \alpha)^2} \chf{[\alpha - 2 \sqrt{\beta}, \alpha + 2 \sqrt{\beta}]}(x) \,dx.
\]

\br
For $b \in \mf{R}$, $1 + c \geq 0$, the free Meixner distributions, normalized to have mean zero and variance one, are
\[
d\mu_{b,c}(x) = \frac{1}{2 \pi} \frac{\sqrt{4 (1 + c) - (x - b)^2}}{1 + b x + c x^2} \,dx + \text{ zero, one, or two atoms}.
\]
They are characterized by their Jacobi parameter sequences having the special form
\[
(0, b, b, b, \ldots), (1, 1+c, 1+c, 1+c, \ldots),
\]
or by the special form of the generating function of their orthogonal polynomials. In particular, $\mu_{0,0} = \SC(0,1)$ is the standard semicircular distribution, $\mu_{b,0}$ are the centered free Poisson distributions, and $\mu_{b,-1}$ are the normalized Bernoulli distributions.

\br
More generally, free Meixner states are states on $\mf{C} \langle \mb{x} \rangle$, characterized by a number of equivalent conditions (see \cite{AnsFree-Meixner}), among them the equations
\begin{equation}
\label{free-quadratic-PDE}
D_i D_j \CumFun{\phi}{\mb{z}} = \delta_{ij} + \sum_{k=1}^d B_{ij}^k D_k \CumFun{\phi}{\mb{z}} + C_{ij} D_i \CumFun{\phi}{\mb{z}} D_j \CumFun{\phi}{\mb{z}}.
\end{equation}
for certain $\set{B_{ij}^k, C_{ij}}$. In \cite{AnsBoolean}, these equations were shown to be equivalent to
\begin{equation}
\label{Boolean-quadratic-PDE}
D_i D_j \eta^\phi(\mb{z}) = \delta_{ij} + \sum_{k=1}^d B_{ij}^k D_k \eta^\phi(\mb{z}) + (1 + C_{ij}) D_i \eta^\phi(\mb{z}) D_j \eta^\phi(\mb{z}).
\end{equation}

\section{Appell polynomials}
\label{Section:Appell}

\subsection{Definition and basic properties}

\begin{Defn}
Let $(\mc{A}, \phi, \psi)$ be an algebra with two functionals. Define the c-free Appell polynomials to be, for each $n \in \mf{N}$, maps
\[
A^{\phi, \psi}: \left(\mc{A}^{sa}\right)^n \rightarrow \mc{A}, \qquad (X_1, X_2, \ldots, X_n) \mapsto \A{\phi, \psi}{X_1, X_2, \ldots, X_n}
\]
by specifying that $\A{\phi, \psi}{X_1, X_2, \ldots, X_n}$ is a polynomial in $X_1, X_2, \ldots, X_n$,
\begin{equation}
\label{Appell-definition-1}
\partial_i \A{\phi, \psi}{X_1, X_2, \ldots, X_n} = \A{\psi}{X_1, \ldots, X_{i-1}} \otimes \A{\phi, \psi}{X_{i+1}. \ldots, X_n}
\end{equation}
with the obvious modifications for $i = 1, n$ corresponding to $\A{\phi, \psi}{\emptyset} = 1$, and
\begin{equation}
\label{Appell-definition-2}
\state{\A{\phi, \psi}{X_1, X_2, \ldots, X_n}} = 0
\end{equation}
for $n \geq 1$. This determines the polynomials uniquely. Here $\A{\psi}{\cdot}$ are the free Appell polynomials for $\psi$ (Section~\ref{Subsection:Free-Appell}).
\end{Defn}

\noindent
Each $\A{\phi, \psi}{\cdot}$ is a multilinear map, and its value is a polynomial in its arguments. In particular, define $\Aa{\vec{u}}{x_1, x_2, \ldots, x_d} \in \mf{C} \langle \mb{x} \rangle$ to be the polynomial such that
\[
\A{\phi, \psi}{X_{u(1)}, X_{u(2)}, \ldots, X_{u(n)}} = \Aa{\vec{u}}{X_1, X_2, \ldots, X_d}.
\]
Note that the polynomial $\Aa{\vec{u}}{\mb{x}}$ depends on the choice of $X_1, X_2, \ldots, X_d$, so in cases where confusion may arise we will denote this polynomial by
\[
A_{\vec{u}}^{X_1, X_2, \ldots, X_d}(x_1, x_2, \ldots, x_d).
\]

\begin{Example}
The low order c-free Appell polynomials are
\begin{align*}
\A{\phi, \psi}{X_1} & = X_1 - \Cum{\phi, \psi}{X_1}, \\
\A{\phi, \psi}{X_1, X_2} & = X_1 X_2 - X_1 \Cum{\phi, \psi}{X_2} - \Cum{\psi}{X_1} X_2 + \Cum{\psi}{X_1} \Cum{\phi, \psi}{X_2} - \Cum{\phi, \psi}{X_1, X_2}, \\
\A{\phi, \psi}{X_1, X_2, X_3} & = X_1 X_2 X_3 - X_1 X_2 \Cum{\phi, \psi}{X_3} - X_1 \Cum{\psi}{X_2} X_3 - \Cum{\psi}{X_1} X_2 X_3 \\
&\quad + X_1 \Cum{\psi}{X_2} \Cum{\phi, \psi}{X_3} - X_1 \Cum{\phi, \psi}{X_2, X_3}  + \Cum{\psi}{X_1} X_2 \Cum{\phi, \psi}{X_3} \\
&\quad + \Cum{\psi}{X_1} \Cum{\psi}{X_2} X_3 - \Cum{\psi}{X_1, X_2} X_3 - \Cum{\psi}{X_1} \Cum{\psi}{X_2} \Cum{\phi, \psi}{X_3} \\
&\quad + \Cum{\psi}{X_1} \Cum{\phi, \psi}{X_2, X_3} + \Cum{\psi}{X_1, X_2} \Cum{\phi, \psi}{X_3} - \Cum{\phi, \psi}{X_1, X_2, X_3}.
\end{align*}
\end{Example}

\begin{Prop}
\label{Prop:Recursions}
Fix $(\mc{A}, \phi, \psi)$.
\begin{enumerate}
\item
For fixed $(X_1, X_2, \ldots, X_d)$, the generating function of their c-free Appell polynomials is
\[
H^{\phi, \psi}(\mb{x}, \mb{z}) = 1 + \sum_{\vec{u}} \Aa{\vec{u}}{\mb{x}} z_{\vec{u}} = \left( 1 - \mb{x} \cdot \mb{z} + \CumFun{\psi}{\mb{z}} \right)^{-1} \left( 1 + \CumFun{\psi}{\mb{z}} - \CumFun{\phi, \psi}{\mb{z}} \right).
\]
\item
The polynomials satisfy a recursion relation
\[
\begin{split}
X \A{\phi, \psi}{X_1, \ldots, X_n}
& = \A{\phi, \psi}{X, X_1, \ldots, X_n} \\
& \quad + \sum_{j=0}^{n-1} \Cum{\psi}{X, X_1, \ldots, X_j} \A{\phi, \psi}{X_{j+1}, \ldots, X_n} + \Cum{\phi, \psi}{X, X_1, \ldots, X_n}.
\end{split}
\]
\item
The monomials have an expansion in terms of the c-free Appell polynomials
\begin{multline*}
X_1 X_2 \ldots X_n
= \sum_{k=0}^n \sum_{\substack{B \subset \set{1, \ldots, n} \\ B = \set{i(1), \ldots, i(k)}}} \prod_{j=1}^k \psi \left[ X_{i(j-1) + 1} \ldots X_{i(j) - 1} \right] \state{X_{i(k) + 1} \ldots X_{u(n)}} \\
\times \A{\phi, \psi}{X_{i(1)}, X_{i(2)}, \ldots, X_{i(k)}}.
\end{multline*}
\item
The explicit formula for the c-free Appell polynomials is
\[
\begin{split}
& \A{\phi, \psi}{X_1, \ldots, X_n} \\
&\qquad = \sum_{\pi \in \Int(n)} \sum_{S \subset \Sing(\pi)} (-1)^{\abs{S^c}} \prod_{\substack{B \in S^c \\ n \not \in B}} \Cum{\psi}{X_i: i \in B} \prod_{\substack{B \in S^c \\ n \in B}}  \Cum{\phi, \psi}{X_i: i \in B} \prod_{\set{i} \in S} X_i.
\end{split}
\]
See Proposition~\ref{Prop:Kailath-Segall} part (c) for notation.
\end{enumerate}
\end{Prop}

\begin{proof}
For part (a), we check that, identifying $\mf{C} \langle \mb{x} \rangle \otimes \mf{C} \langle \mb{x} \rangle = \mf{C} \langle \mb{x}, \mb{y} \rangle$,
\[
\begin{split}
\partial_{x_i} H^{\phi, \psi}(\mb{x}, \mb{z})
& = \left( 1 - \mb{x} \cdot \mb{z} + \CumFun{\psi}{\mb{z}} \right)^{-1} z_i \left( 1 - \mb{y} \cdot \mb{z} + \CumFun{\psi}{\mb{z}} \right)^{-1} \left( 1 + \CumFun{\psi}{\mb{z}} - \CumFun{\phi, \psi}{\mb{z}} \right) \\
& = H^\psi(\mb{x}, \mb{z}) z_i H^{\phi, \psi}(\mb{y}, \mb{z}),
\end{split}
\]
where
\[
H^\psi(\mb{x}, \mb{z}) = \left( 1 - \mb{x} \cdot \mb{z} + \CumFun{\psi}{\mb{z}} \right)^{-1} = 1 + \sum_{\vec{u}} A^{\psi}_{\vec{u}} \left( \mb{x} \right) z_{\vec{u}}
\]
is the generating function for the free Appell polynomials of $\psi$. Also, making the change of variables \eqref{z-M-w} and using relation \eqref{C-cumulant-identity},
\begin{equation}
\label{H-reduction}
\begin{split}
\state{H^{\phi, \psi}(\mb{x}, \mb{z})}
& = \state{\left( 1 - \mb{x} \cdot \mb{z} + \CumFun{\psi}{\mb{z}} \right)^{-1}} \left( 1 + \CumFun{\psi}{\mb{z}} - \CumFun{\phi, \psi}{\mb{z}}  \right) \\
& = \state{\left( 1 - (1 + M^\psi(\mb{w})) \mb{x} \cdot \mb{w} + M^\psi(\mb{w}) \right)^{-1}} \\
& \quad \times (1 + M^\psi(\mb{w})) (1 + M^\phi(\mb{w}))^{-1} \\
& = \state{(1 - \mb{x} \cdot \mb{w})^{-1}} (1 + M^\psi(\mb{w}))^{-1} (1 + M^\psi(\mb{w})) (1 + M^\phi(\mb{w}))^{-1} \\
& = (1 + M^\phi(\mb{w})) (1 + M^\phi(\mb{w}))^{-1}= 1.
\end{split}
\end{equation}
It follows that the coefficients in the power series expansion of $H^{\phi, \psi}(\mb{x}, \mb{z})$ satisfy conditions \eqref{Appell-definition-1} and \eqref{Appell-definition-2}, which determine these coefficients uniquely.

\br
For part (b), we observe that
\[
\begin{split}
1 + \CumFun{\psi}{\mb{z}} - \CumFun{\phi, \psi}{\mb{z}}
& = \left( 1 - \mb{x} \cdot \mb{z} + \CumFun{\psi}{\mb{z}} \right) H^{\phi, \psi}(\mb{x}, \mb{z}) \\
& = \left( 1 - \mb{x} \cdot \mb{z} + \CumFun{\psi}{\mb{z}} \right) \left(1 + \sum_{\vec{u}} \Aa{\vec{u}}{\mb{x}} z_{\vec{u}} \right)
\end{split}
\]
or
\[
- \CumFun{\phi, \psi}{\mb{z}}
= - \sum_i x_i z_i + \sum_{\vec{u}} \Aa{\vec{u}}{\mb{x}} z_{\vec{u}} - \sum_{i, \vec{u}} x_i \Aa{\vec{u}}{\mb{x}} z_i z_{\vec{u}} + \sum_{\vec{u}} \Aa{\vec{u}}{\mb{x}} \CumFun{\psi}{\mb{z}} z_{\vec{u}}.
\]
Identifying coefficients of $z_i z_{\vec{u}}$ with $\abs{\vec{u}} = n$, we get
\[
x_i \Aa{\vec{u}}{\mb{x}}
= \Aa{(i, \vec{u})}{\mb{x}} + \sum_{j=0}^{n-1} \Cum{\psi}{x_i x_{u(1)} \ldots x_{u(j)}} \Aa{(u(j+1), \ldots, u(n))}{\mb{x}} + \Cum{\phi, \psi}{x_i x_{\vec{u}}}.
\]

\br
For part (c),
\[
H^{\phi, \psi}(\mb{x}, (1 + M^\psi(\mb{w})) \mb{w}) = (1 - \mb{x} \cdot \mb{w})^{-1} (1 + M^\phi(\mb{w}))^{-1},
\]
so that
\[
(1 - \mb{x} \cdot \mb{w})^{-1} = \left( 1 + \sum_{\vec{v}} \Aa{\vec{v}}{\mb{x}} [(1 + M^\psi(\mb{w})) \mb{w}]_{\vec{v}} \right) (1 + M^\phi(\mb{w})).
\]
It follows that for $\abs{\vec{u}} = n$,
\[
x_{\vec{u}} = \sum_{k=0}^n \sum_{\substack{B \subset \set{1, \ldots, n} \\ B = \set{i(1), \ldots, i(k)}}} \prod_{j=1}^k \psi \left[ x_{u(i(j-1) + 1)} \ldots x_{u(i(j) - 1)} \right] \state{x_{u(i(k) + 1)} \ldots x_{u(n)}} \Aa{(\vec{u}:B)}{\mb{x}}.
\]

\br
Part (d) is obtained by combining parts (c) and (d) of Proposition~\ref{Prop:Kailath-Segall}.
\end{proof}

\begin{Prop}
\label{Prop:c-Appell-characterization}
For any $\phi, \psi$, the c-free Appell polynomials are the unique polynomial family satisfying
\[
(\psi \otimes I) \partial_i P\left(X_1, X_2, \ldots, X_n\right)
= \delta_{i1} P\left(X_{2}, \ldots, X_n \right)
\]
and
\[
\state{P\left(X_1, X_2, \ldots, X_n \right)} = 0.
\]
\end{Prop}

\br
The proof is similar to Proposition~\ref{Prop:Appell-characterization}.

\begin{Cor}
The Boolean Appell polynomials of \cite{AnsBoolean} are, in agreement with part (b) of Example~\ref{Example:Conditional-freeness}, the c-free Appell polynomials corresponding to $\psi = \delta_0$.
\end{Cor}

\begin{proof}
For any monomial,
\[
\begin{split}
(\delta_0 \otimes I) \partial_i (X_{u(1)} X_{u(2)} \ldots X_{u(n)})
& = \sum_{j: u(j) = i} (\delta_0 \otimes I) \bigl((X_{u(1)} \ldots X_{u(j-1)}) \otimes (X_{u(j+1)} \ldots X_{u(n)})\bigr) \\
& = \delta_{i, u(1)} X_{u(2)} \ldots X_{u(n)}
= D_i (X_{u(1)} X_{u(2)} \ldots X_{u(n)}).
\end{split}
\]
So polynomials satisfying the conditions in the preceding proposition with $\psi = \delta_0$ are exactly those satisfying the definition in Section 3.2 of \cite{AnsBoolean}.
\end{proof}

\subsection{Orthogonality}
\label{Subsection:Orthogonal}
The argument in equation~\eqref{H-reduction} shows that for any $\psi$, the c-free Appell polynomials are also Boolean (and hence free) Sheffer polynomials for $\phi$:
\[
H^{\phi, \psi}(\mb{x}, \mb{z}) = (1 - \mb{x} \cdot \mb{V}(\mb{z}))^{-1} (1 - \eta^\phi(\mb{V}(\mb{z}))) ,
\]
where
\[
V_i(\mb{z}) = (1 + \CumFun{\psi}{\mb{z}})^{-1} z_i.
\]
In one variable, every Boolean Sheffer family of $\phi$ is of this form, but this is not the case in several variables because of the special form of the series $\mb{V}$. By Proposition~7 of \cite{AnsBoolean}, these polynomials are orthogonal if and only if
\[
(D_i \eta^\phi)(\mb{V}(\mb{z})) = z_i
\]
and $\phi$ is a free Meixner state. Lemma~\ref{Lemma:Quadratic} and Theorem~\ref{Thm:Quadratic} below separate these two conditions and describe a number of other properties equivalent to them. The orthogonality is with respect to the state $\phi$.

\begin{Lemma}
\label{Lemma:Quadratic}
The following are equivalent.
\begin{enumerate}
\item
The two-state free cumulant generating function of $(\phi, \psi)$ is simple quadratic:
\[
\CumFun{\phi, \psi}{\mb{z}} = \sum_{i=1}^d z_i^2.
\]
\item
$\phi = \State{\psi}$.
\item
All of the c-free Appell polynomials for $(\phi, \psi)$  are orthogonal to the degree one c-free Appell polynomials, and $\phi$ has mean zero and identity covariance.
\end{enumerate}
\end{Lemma}

\begin{proof}
\[
\sum_{i=1}^d (1 + M^\psi(\mb{w})) w_i (1 + M^\psi(\mb{w})) w_i = \CumFun{\phi, \psi}{(1 + M^\psi(\mb{w})) \mb{w}} = (1 + M^\psi(\mb{w})) \eta^\phi(\mb{w})
\]
if and only if
\[
\eta^\phi(\mb{w}) = \sum_{i=1}^d w_i (1 + M^\psi(\mb{w})) w_i
\]
so that $\phi = \State{\psi}$. Thus (a) $\Leftrightarrow$ (b).

\br
It follows from the recursion relation in part (b) of Proposition~\ref{Prop:Recursions} that
\[
\begin{split}
\state{\A{\phi, \psi}{X} \A{\phi, \psi}{X_1, X_2, \ldots, X_n}}
& = \state{(X - \state{X}) \A{\phi, \psi}{X_1, X_2, \ldots, X_n}} \\
& = \Cum{\phi, \psi}{X, X_1, \ldots, X_n}.
\end{split}
\]
So the degree one c-free Appell polynomials are orthogonal to the rest if and only if
\[
\Cum{\phi, \psi}{x_{u(1)}, x_{u(2)}, \ldots, x_{u(n)}} = 0
\]
for $n > 2$ and $\Cum{\phi, \psi}{x_i, x_j} = 0$ for $i \neq j$, so that $\CumFun{\phi, \psi}{\mb{z}} = \sum_{i=1}^d (a_i z_i + b_i z_i^2)$. The normalization of $\phi$ forces $\CumFun{\phi, \psi}{\mb{z}} = \sum_{i=1}^d z_i^2$.
\end{proof}

\begin{Remark}
A free version of the mapping $\Phi$ was considered in the last section of \cite{AnsMulti-Sheffer}; the states described there are the image under the appropriate Bercovici-Pata bijection of the states in the next theorem.
\end{Remark}

\begin{Thm}
\label{Thm:Quadratic}
Let $\phi, \psi$ be states with MOPS, $\phi$ with mean zero and identity covariance. The c-free Appell polynomials $A^{\phi, \psi}$ are orthogonal if and only if either of the following equivalent conditions holds:
\begin{enumerate}
\item
$R^{\phi, \psi}$ and $R^\psi$ are both quadratic,
\[
\CumFun{\phi, \psi}{\mb{z}} = \sum_{i=1}^d z_i^2
\]
and
\[
\CumFun{\psi}{\mb{z}} = \sum_{i=1}^d b_i z_i + \sum_{i=1}^d (1 + c_i) z_i^2.
\]
\item
$\psi$ is the joint distribution of freely independent semicircular elements with means $b_i$ and variances $1 + c_i$, $b_i \in \mf{R}$, $c_i \geq -1$, and
\[
\phi = \State{\psi}.
\]
\end{enumerate}
In any case, $\phi$ is a free Meixner state, and
\[
D_i D_j \eta^\phi = \delta_{ij} + b_i D_j \eta^\phi + (1 + c_i) D_i \eta^\phi D_j \eta^\phi.
\]
\end{Thm}

\begin{proof}
The main result (Theorem 2) of \cite{AnsMonic} states that polynomials with respect to a state with MOPS are orthogonal if any only if they satisfy a recursion relation involving polynomials of only three consecutive total degrees. Comparing that result with the recursion in part (b) of Proposition~\ref{Prop:Recursions}, this is the case exactly when
\[
\Cum{\phi, \psi}{X_1, X_2, \ldots, X_n} = \Cum{\psi}{X_1, X_2, \ldots, X_n} = 0
\]
for $n \geq 3$, so that both the cumulant functions are quadratic.

\br
For the change of variable \eqref{z-M-w}, the generating function for the c-free Appell polynomials is
\[
H^{\phi, \psi}(\mb{x}, \mb{z}) = (1 - \mb{x} \cdot \mb{w})^{-1} (1 + M^\phi(\mb{w}))^{-1} = (1 - \mb{x} \cdot \mb{w})^{-1} (1 - \eta^\phi(\mb{w})).
\]
By Proposition 7 of \cite{AnsBoolean}, the polynomials with such a generating function are orthogonal if and only if $\phi$ is a free Meixner state and
\[
D_i \eta^\phi(\mb{w}) = z_i.
\]
That is,
\[
\eta^\phi(\mb{w}) = \sum_i w_i z_i = \sum_i w_i (1 + M^\psi(\mb{w})) w_i
\]
and $\phi = \State{\psi}$.

\br
It remains to show that the only $\psi$ such that $\State{\psi}$ is a free Meixner state are the ones in the statement of the theorem. Indeed,
\[
D_j \eta^\phi(\mb{w}) = (1 + M^\psi(\mb{w})) w_j
\]
and so
\[
D_i D_j \eta^\phi(\mb{w}) = \delta_{ij} + D_i M^\psi(\mb{w}) w_j.
\]
On the other hand,
\begin{equation}
\label{D-eta-D-R-M-w}
D_i \eta^{\phi}(\mb{w})
= D_i \Bigl( (1 + M^\psi(\mb{w}))^{-1} \CumFun{\phi, \psi}{(1 + M^\psi(\mb{w})) \mb{w}} \Bigr)
= (D_i R^{\phi, \psi}) \left((1 + M^\psi(\mb{w})) \mb{w}\right)
\end{equation}
and using equation~\eqref{R-w-M},
\[
D_i M^\psi(\mb{w}) = (1 + M^\psi(\mb{w})) (D_i R^\psi) \left(\mb{w} (1 + M^\psi(\mb{w}))\right).
\]
It follows that if $\phi$ is a free Meixner state, with $\eta^\phi$ satisfying equation~\eqref{Boolean-quadratic-PDE}, then
\[
\begin{split}
D_i D_j \eta^\phi(\mb{w})
& = \delta_{ij} + D_i M^\psi(\mb{w}) w_j \\
& = \delta_{ij} + \sum_k B_{ij}^k (1 + M^\psi(\mb{w})) w_k + (1 + C_{ij}) (1 + M^\psi(\mb{w})) w_i (1 + M^\psi(\mb{w})) w_j \\
& = \delta_{ij} + (1 + M^\psi(\mb{w})) (D_i R^\psi) \left(\mb{w} (1 + M^\psi(\mb{w}))\right) w_j.
\end{split}
\]
Therefore for any $i, j$,
\[
\sum_k B_{ij}^k w_k + (1 + C_{ij}) w_i (1 + M^\psi(\mb{w})) w_j
= (D_i R^\psi) \left(\mb{w} (1 + M^\psi(\mb{w}))\right) w_j.
\]
It follows that $B_{ij}^k = \delta_{jk} b_{ij}$ and for all $j$,
\[
b_{ij} + (1 + C_{ij}) z_i
= D_i \CumFun{\psi}{\mb{z}}
\]
for $z_i = w_i (1 + M^\psi(\mb{w}))$, so that
\[
\CumFun{\psi}{\mb{z}} = \sum_i b_{ij} z_i + \sum_i (1 + C_{ij}) z_i^2.
\]
This is true for any $j$, therefore $b_{ij} = b_i$, $C_{ij} = c_i$, and finally
\[
\CumFun{\psi}{\mb{z}} = \sum_i \left(b_i z_i + (1 + c_i) z_i^2 \right),
\]
so that $\psi$ is a free product of semicircular distributions with means $b_i$ and variances $(1 + c_i)$.
\end{proof}

\begin{Example}
\label{Example:Appell-second-kind}
Similarly to Proposition~\ref{Prop:c-Appell-characterization}, the c-free Appell polynomials satisfy, and are characterized by, the properties
\[
(I \otimes \phi) \partial_i \A{\phi, \psi}{X_1, X_2, \ldots, X_n} = \delta_{i n} \A{\psi}{X_1, \ldots, X_{n-1}}
\]
and
\[
\state{\A{\phi, \psi}{X_1, X_2, \ldots, X_n}} = 0.
\]
In particular, if $\A{\phi, \psi}{X_1, X_2, \ldots, X_n}$ are orthogonal with respect to $\phi$, then $\A{\psi}{X_1, \ldots, X_n}$ are the corresponding orthogonal polynomials of the second kind. Theorem~\ref{Thm:Quadratic}, combined with Proposition 3.18 of \cite{AnsAppell}, shows that the free Appell polynomials $A^{\psi}$ are orthogonal if and only if $\psi$ is a free product of semicircular distributions. Thus, $A^\psi$ are orthogonal if and only if they are the polynomials of the second kind for $\phi$ and $A^{\phi, \psi}$ are orthogonal.
\end{Example}

\begin{Remark}[Conjugate variables in two-state free probability theory]
For
\[
X_1, X_2, \ldots, X_d \in (\mc{A}, \phi, \psi),
\]
we can define their (formal) c-free conjugate variables $\xi_i$ via
\[
(\psi \otimes \phi) [\partial_i P(X_1, X_2, \ldots, X_d)] = \state{\xi_i P(X_1, X_2, \ldots, X_d)}.
\]
In particular
\[
\state{\xi_i \A{\phi, \psi}{X_1, X_2, \ldots, X_d}} = (\psi \otimes \phi) [\partial_i \A{\phi, \psi}{X_1, X_2, \ldots, X_d}] = \delta_{i 1} \delta_{d 1}
\]
and more generally
\[
\state{\xi_i A^{X_1, X_2, \ldots, X_d}_{\vec{u}}(X_1, X_2, \ldots, X_d)} = \delta_{i, \vec{u}}.
\]
In one variable, if $\state{p} = \int_{\mf{R}} p(x) \,d\mu(x)$, $\psi[p] = \int_{\mf{R}} p(x) \,d\nu(x)$,
%with $\mu, \nu$ compactly supported with densities in $L^3(\mf{R}, dx)$ and $\nu \ll \mu$ with $\frac{d\nu}{d\mu} \in L^\infty(\mu)$, then
under appropriate conditions on $\mu, \nu$
\[
\xi(x) = \pi \left( H\nu(x) + H\mu(x) \frac{d\nu}{d\mu}(x) \right),
\]
where $H\nu(x) = \frac{1}{\pi} \int_{\mf{R}} \frac{1}{x - y} \,d\nu(y)$ is the Hilbert transform. Without a random matrix connection, the use of these objects at this point is unclear, so we list only one basic result. It elucidates the connection between minimization of the free Fisher information on free semicircular variables, and the orthogonality of the Chebyshev polynomials. In particular, it implies a version of Proposition~6.9 of \cite{Voi-Entropy5}.
\end{Remark}

\begin{Prop}
\label{Prop:Fisher}
$X_1, X_2, \ldots, X_d \in (\mc{A}, \phi, \psi)$ with well-defined formal c-free conjugate variables satisfy
\[
\state{\sum_{i=1}^d \xi_i^2} \state{\sum_{i=1}^d X_i^2} \geq d^2.
\]
Denoting their joint distributions $\mu = \phi^{X_1, \ldots, X_d}$ and $\nu = \psi^{X_1, \ldots, X_d}$, the equality is achieved exactly for $\mu = \State{\nu}^{\uplus \lambda}$ (here $\uplus$ is the Boolean convolution, see Section~\ref{Subsection:Convolutions}).
\end{Prop}

\begin{proof}
Since
\[
\sum_{i=1}^d \state{\A{\phi, \psi}{X_i}^2} = \sum_{i=1}^d \state{(X_i - \state{X_i})^2} = \sum_{i=1}^d \Var{X_i},
\]
it follows that
\[
\begin{split}
\state{\sum_{i=1}^d \xi_i^2}
& = \norm{(\xi_1, \xi_2, \ldots, \xi_d)}_{\phi}^2 \\
& \geq \frac{\ip{(\xi_1, \xi_2, \ldots, \xi_d)}{(\A{\phi, \psi}{X_1}, \A{\phi, \psi}{X_1}, \ldots, \A{\phi, \psi}{X_d})}_\phi^2}{\norm{(\A{\phi, \psi}{X_1}, \A{\phi, \psi}{X_1}, \ldots, \A{\phi, \psi}{X_d})}_{\phi}^2} \\
& = \frac{\Bigl(\sum_{i=1}^d \state{\xi_i \A{\phi, \psi}{X_i}} \Bigr)^2}{\sum_{i=1}^d \state{\A{\phi, \psi}{X_i}^2}}
= \frac{d^2}{\sum_{i=1}^d \Var{X_i}}.
\end{split}
\]
Therefore
\[
\state{\sum_{i=1}^d \xi_i^2} \state{\sum_{i=1}^d X_i^2} \geq \state{\sum_{i=1}^d \xi_i^2} \sum_{i=1}^d \Var{X_i} \geq d^2.
\]
The equality is achieved if and only if each $X_i$ is $\phi$-centered and $(\xi_1, \xi_2, \ldots, \xi_d)$ is a multiple of
\[
(\A{\phi, \psi}{X_1}, \A{\phi, \psi}{X_2}, \ldots, \A{\phi, \psi}{X_d}).
\]
Using the proof of Lemma~\ref{Lemma:Quadratic}, this says that $\CumFun{\phi, \psi}{\mb{z}} = \lambda \sum_{i=1}^d z_i^2$, in other words
\[
\eta^\mu(\mb{w}) = \lambda \sum_{i=1}^d w_i(1 + M^\nu(\mb{w})) w_i = \lambda \eta^{\State{\nu}}
\]
and $\mu = \State{\nu}^{\uplus \lambda}$.
\end{proof}

\section{Fock space representation and processes}
\label{Section:Processes}

\subsection{Fock space representation}
\label{Subsection:Fock}
Let $\mc{A}_0$ be an algebra without identity, and $\mu, \nu$ functionals on it. Let $\mc{A}$ be the unital algebra generated by non-commuting symbols $\set{X(f) : f \in \mc{A}^{sa}_0}$ subject to the linearity relations
\[
X(\alpha f + \beta g) = \alpha X(f) + \beta X(g).
\]
Equivalently, $\mc{A}$ is the tensor algebra of $\mc{A}_0$. The star-operation on it is determined by the requirement that all $X(f)$, $f \in \mc{A}^{sa}_0$ are self-adjoint. For such $f, f_i$, define the c-free Kailath-Segall polynomials to be multilinear maps $W$ from $\mc{A}^{sa}_0$ to $\mc{A}$ determined by
\begin{gather*}
X(f) = W(f) + \mu[f], \\
X(f) W(f_1) = W(f, f_1) + W(f f_1) + \mu[f f_1] + \nu[f] W(f_1), \\
\begin{split}
X(f) W(f_1, f_2, \ldots, f_n)
& = W(f, f_1, f_2, \ldots, f_n) + W(f f_1, f_2, \ldots, f_n) \\
&\quad + \nu[f f_1] W(f_2, \ldots, f_n) + \nu[f] W(f_1, f_2, \ldots, f_n).
\end{split}
\end{gather*}
Denoting
\[
f_\Lambda = \prod_{i \in \Lambda} f_i,
\]
$W(f_1, f_2, \ldots, f_n)$ is a polynomial in $\set{X(f_\Lambda): \Lambda \subset \set{1, 2, \ldots, n}}$.

\br
In the case that $\mu, \nu$ are positive (semi-)definite, the c-free Kailath-Segall polynomials have a representation on the Fock space
\[
\mf{C} \Omega \oplus L^2(\mc{A}_0, \mu) \oplus \bigoplus_{n=1}^\infty \left( L^2(\mc{A}_0, \nu)^{\otimes n} \otimes L^2(\mc{A}_0, \mu) \right)
\]
via
\begin{equation}
\label{Fock-representation}
W(f_1, f_2, \ldots, f_n) \Omega = f_1 \otimes f_2 \otimes \ldots \otimes f_n,
\end{equation}
so that for $f \in \mc{A}^{sa}_0$, $X(f)$ is a symmetric operator.

\begin{Example}
We recover the limiting distributions in the central and Poisson limit theorems of \cite{BLS96} as follows. Let
\[
\mc{A}_0 = \set{P(x) \in \mf{C}[x] | P(0) = 0}.
\]
In the analog of the Gaussian representation, we take $\mu[x] = \nu[x] = 0$, skip the second term in the recursions, and quotient out by the relation $x^2 - 1$ (see Example 4.14 of \cite{AnsAppell} for more details). Then each $W$ is a polynomial in $x$, with recursions
\begin{align*}
x & = W_1(x), \\
x W_1(x) & = W_2(x) + \mu[x^2], \\
x W_n(x) & = W_{n+1}(x) + \nu[x^2] W_{n-1}(x).
\end{align*}
Similarly, for the analog of the Poisson distribution, we take $\mu[x^i] = \mu[x^j]$, $\nu[x^i] = \nu[x^j]$ for all $i, j$, and quotient out by the relation $x^2 - x$. Again each $W$ is a polynomial in $x$, with recursions
\begin{align*}
x & = W_1(x) + \mu[x], \\
x W_1(x) & = W_2(x) + (1 + \nu[x]) W_1(x) + \mu[x], \\
x W_n(x) & = W_{n+1}(x) + (1 + \nu[x]) W_n(x) + \nu[x] W_{n-1}(x).
\end{align*}
\end{Example}

\begin{Prop}
\label{Prop:Kailath-Segall}
For $S \subset \Outer(\pi)$, $C \in \Outer(\pi) \backslash S$, denote
\[
C < S \text{ if } \exists c \in C, B \in S, b \in B: c < b
\]
and
\[
C > S \text{ if } \forall c \in C, B \in S, b \in B: c > b.
\]
\begin{enumerate}
\item
The monomials have an expansion in terms of the c-free Kailath-Segall polynomials:
\begin{multline*}
X(f_1) \ldots X(f_n) = \sum_{k=1}^n \sum_{\substack{\pi \in \NC(n) \\ \Outer(\pi) = \set{B_1, \ldots, B_k}}} \sum_{S \subset \Outer(\pi)} \prod_{C \in \Inner(\pi)} \nu[f_C] \prod_{\substack{C \in \Outer(\pi) \backslash S \\ C < S}} \nu[f_C] \\
\times \prod_{\substack{C \in \Outer(\pi) \backslash S \\ C > S}} \mu[f_C] \quad W(f_{B}: B \in S).
\end{multline*}
Pictorially, the outer classes $C$ with $C < S$ are ``potentially inner'', since they will become inner if the ``open'' classes of $S$ are closed with something to the left of them. See Figure~\ref{Figure:X} for an example.

\begin{figure}[hhh]
\psfig{figure=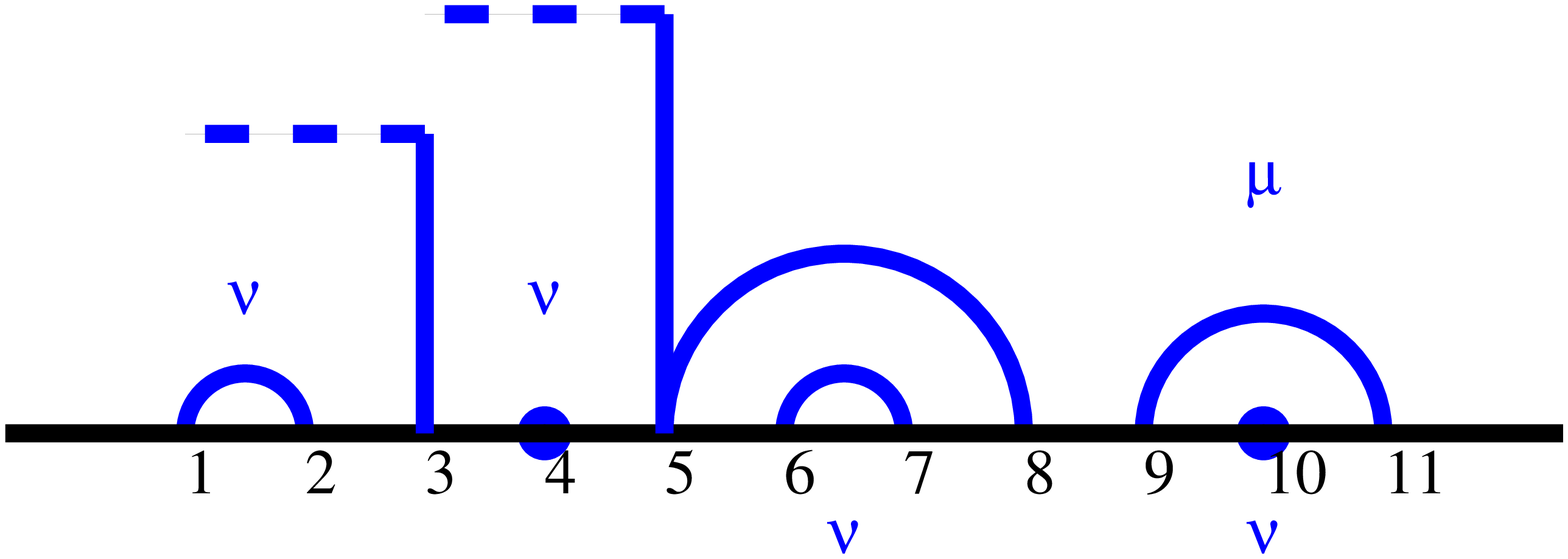,height=0.2\textwidth,width=0.6\textwidth}
\caption{\label{Figure:X}
Graphical representation of the term
\[
\nu[f_6 f_7] \ \nu[f_{10}] \ \nu[f_1 f_2] \ \nu[f_4] \ \mu[f_9 f_{11}] \ W(f_3, f_5 f_8),
\]
with $\pi = \set{(1,2), (3), (4), (5,8), (6,7), (9,11), (10)}$ and $S = \set{(3), (5,8)}$.}
\end{figure}

\item
In the Fock space representation \eqref{Fock-representation},
\[
\mu[f_1 f_2 \ldots f_n] = \Cum{\phi, \psi}{X(f_1), X(f_2), \ldots, X(f_n)}
\]
and
\[
\nu[f_1 f_2 \ldots f_n] = \Cum{\psi}{X(f_1), X(f_2), \ldots, X(f_n)}
\]
for
\[
\state{X(f_1) X(f_2) \ldots X(f_n)} = \ip{\Omega}{X(f_1) \ldots X(f_n) \Omega}
\]
and
\[
\psi\left[ X(f_1) \ldots X(f_n) \right] = \sum_{\pi \in \NC(n)} \prod_{C \in \pi} \nu[f_C].
\]
It follows that if $\set{f_i}$ are mutually orthogonal, meaning $f_i f_j = 0$ for $i \neq j$, then $\set{X(f_i)}$ are freely independent with respect to $\psi$ and c-free with respect to $(\phi, \psi)$ (and also $(\phi | \psi)$-free).
\item
The c-free Kailath-Segall polynomials are
\[
W(f_1, f_2, \ldots, f_n) = \sum_{\pi \in \Int(n)} \sum_{S \subset \Sing(\pi)} (-1)^{n - \abs{S^c}} \prod_{\substack{\set{i} \in S \\ i \neq n}} \nu[f_i] \prod_{\set{n} \in S} \mu[f_n] \prod_{B \in S^c} X(f_B),
\]
where $\Sing(\pi)$ are all the singletons (one-element classes) of $\pi$, and $\prod_{\set{n} \in S} \mu[f_n] = 1$ if $\set{n} \not \in S$. See Figure~\ref{Figure:W} for an example.

\begin{figure}[hhh]

\psfig{figure=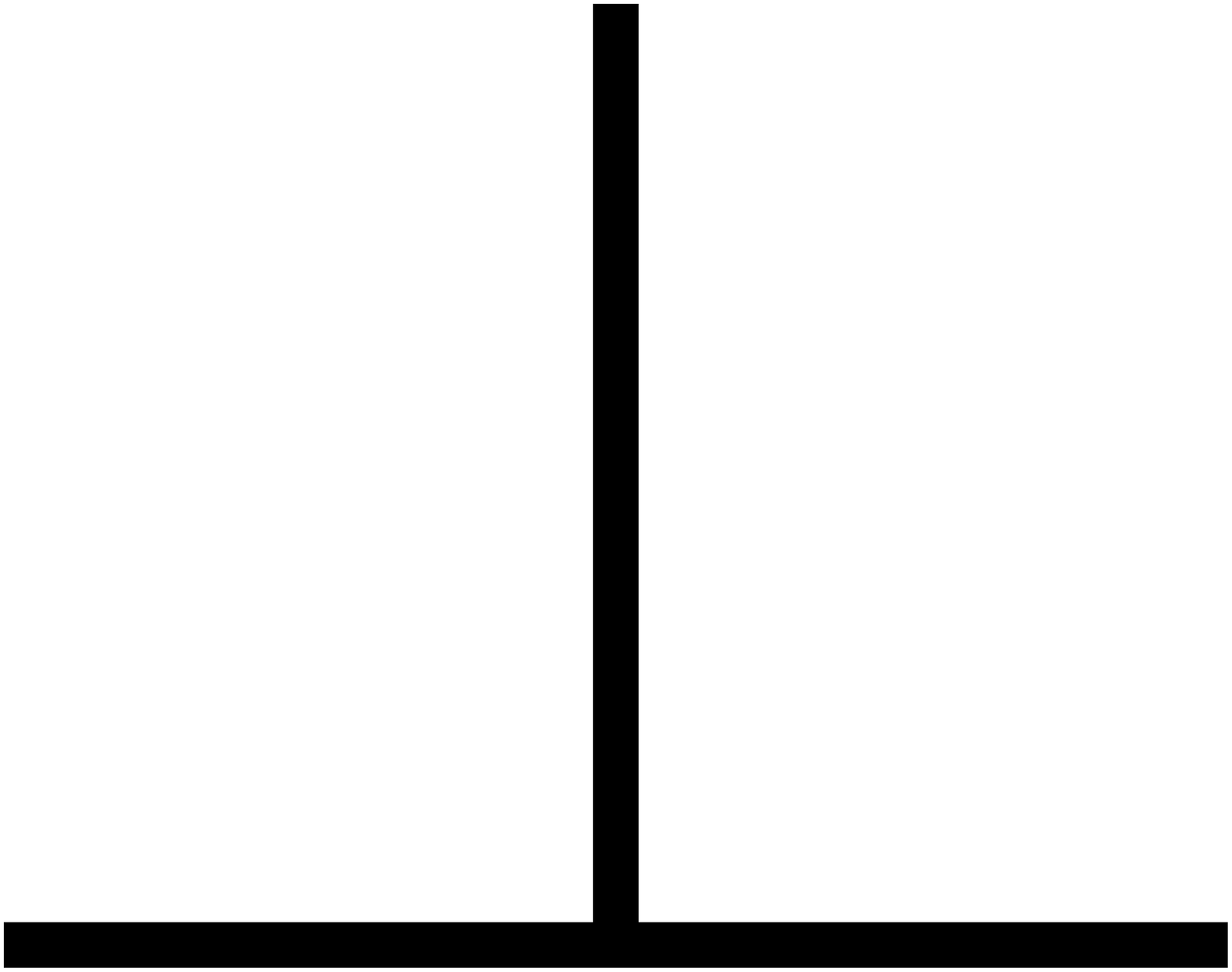,height=0.08\textwidth,width=0.08\textwidth} \psfig{figure=one1b.eps,height=0.08\textwidth,width=0.08\textwidth} - \psfig{figure=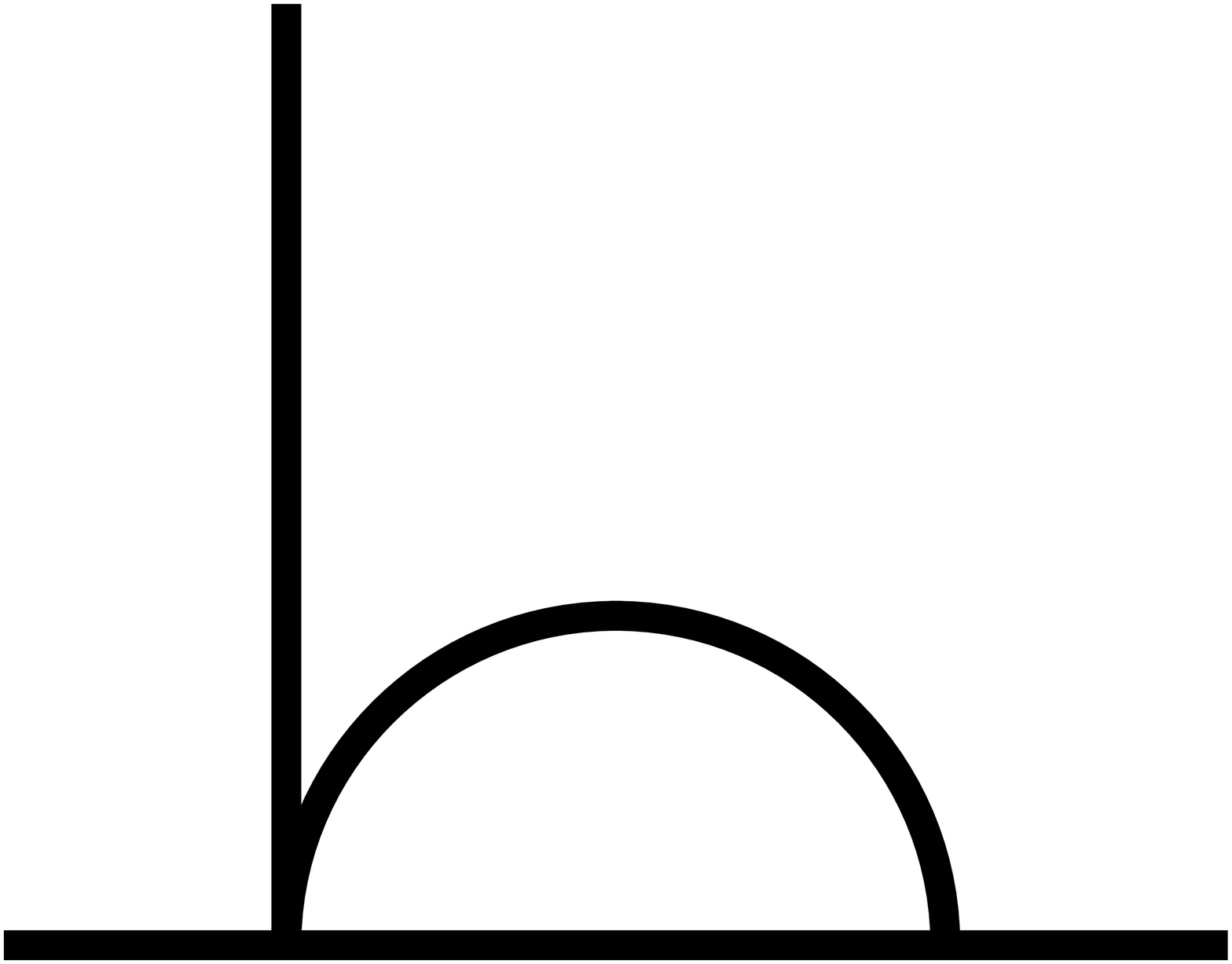,height=0.08\textwidth,width=0.16\textwidth} - \psfig{figure=one1b.eps,height=0.08\textwidth,width=0.08\textwidth} \psfig{figure=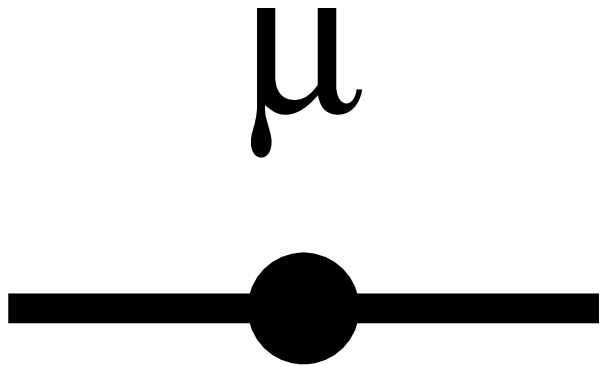,height=0.05\textwidth,width=0.08\textwidth} - \psfig{figure=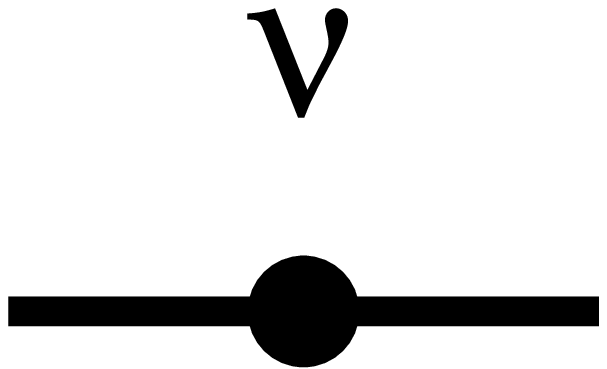,height=0.05\textwidth,width=0.08\textwidth} \psfig{figure=one1b.eps,height=0.08\textwidth,width=0.08\textwidth} + \psfig{figure=one2b-nu.eps,height=0.05\textwidth,width=0.08\textwidth} \psfig{figure=one2b-mu.eps,height=0.05\textwidth,width=0.08\textwidth}
\caption{\label{Figure:W}
$W(f_1, f_2) = X(f_1) X(f_2) - X(f_1 f_2) - X(f_1) \mu[f_2] - \nu[f_1] X(f_2) + \nu[f_1] \mu[f_2]$.}
\end{figure}

\item
The c-free Appell polynomials have an expansion in terms of the c-free Kailath-Segall polynomials:
\[
\A{\phi, \psi}{X(f_1), X(f_2), \ldots, X(f_n)}
= \sum_{\substack{\pi \in \Int(n) \\ \pi = (B_1, B_2, \ldots, B_k)}} W(f_{B_1}, f_{B_2}, \ldots, f_{B_k}).
\]
In particular, this linear combination of the Kailath-Segall polynomials is in fact a polynomial in $\set{X(f_1), X(f_2), \ldots, X(f_n)}$ only. See Figure~\ref{Figure:Appell} for an example.

\begin{figure}[hhh]

\psfig{figure=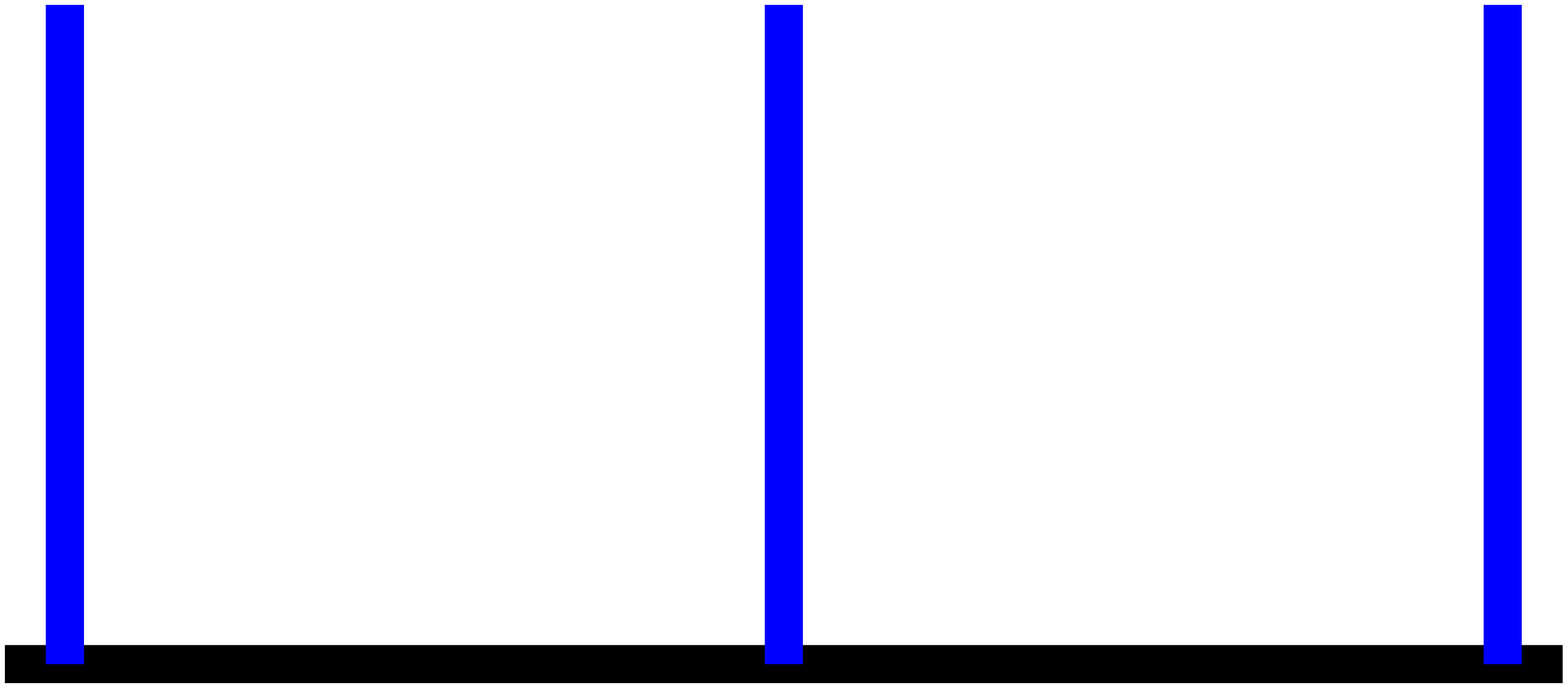,height=0.1\textwidth,width=0.2\textwidth} + \psfig{figure=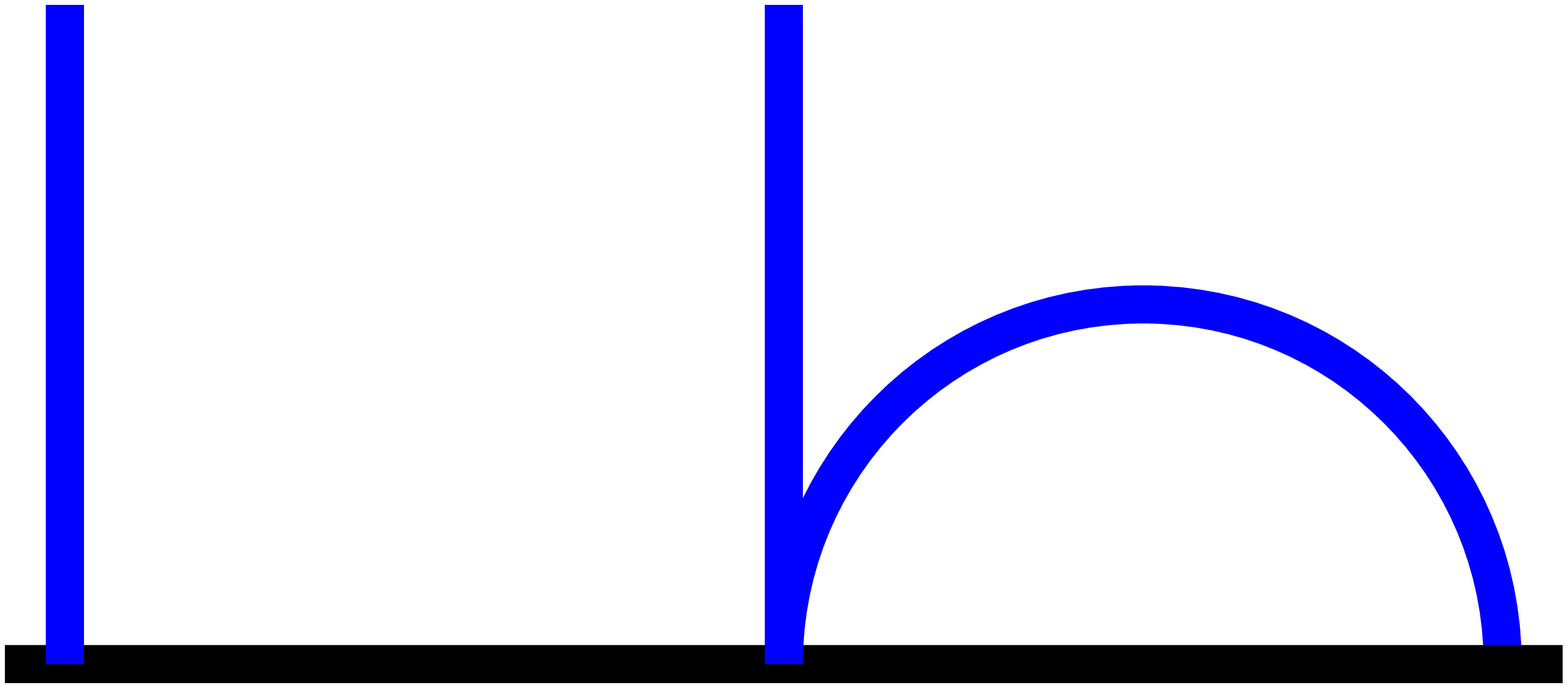,height=0.1\textwidth,width=0.2\textwidth}+ \psfig{figure=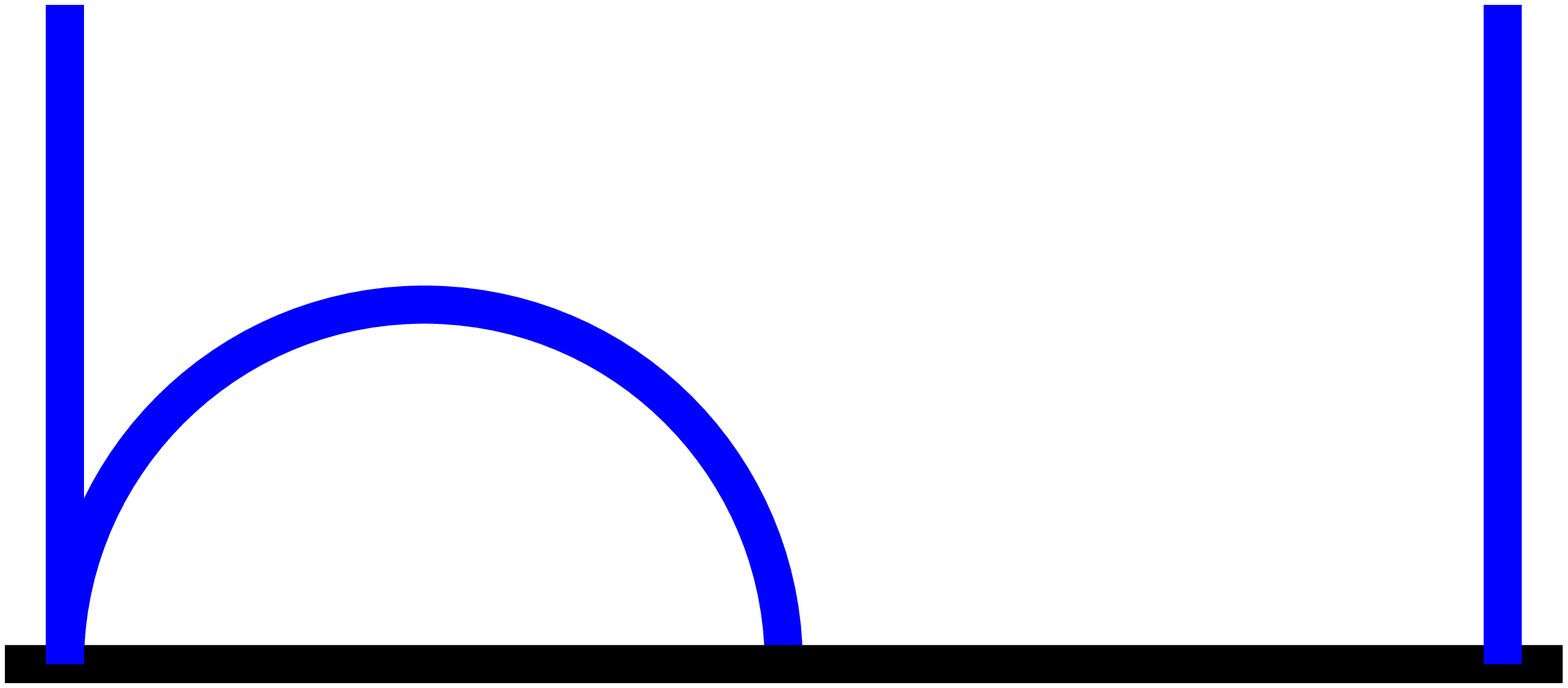,height=0.1\textwidth,width=0.2\textwidth} + \psfig{figure=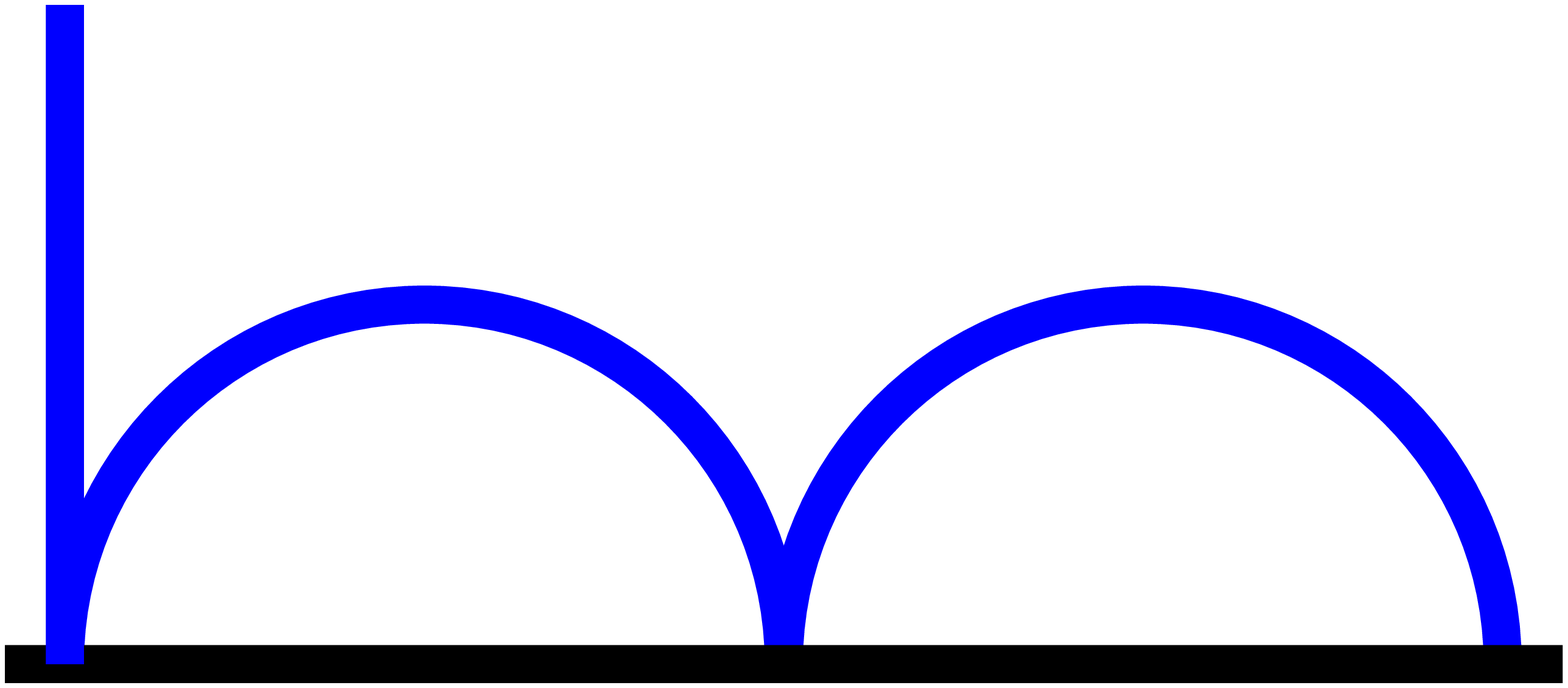,height=0.1\textwidth,width=0.2\textwidth}
\caption{\label{Figure:Appell}
$\A{\phi, \psi}{X(f_1), X(f_2), X(f_3)}$ expanded as $W(f_1, f_2, f_3) + W(f_1, f_2 f_3) + W(f_1 f_2, f_3) + W(f_1 f_2 f_3)$.}
\end{figure}

\end{enumerate}
\end{Prop}

\begin{proof}
The proofs are by induction and very similar to the calculations in \cite{AnsAppell}, so we only outline them. For part (a), multiplying the sum by $X(f_0)$ on the left and applying the recursions results in the following transformations of the pair $(\pi, S)$: for $B$ the left-most class of $S$,
\[
\begin{split}
\text{If } S = \emptyset, \quad & \pi \mapsto \set{0} \cup \pi, \quad S \mapsto \set{0} \text{ or } S \mapsto S, \\
\text{If } \abs{S} \geq 1, \quad & \pi \mapsto \set{0} \cup \pi, \quad S \mapsto \set{0} \cup S \text{ or } S \mapsto S, \\
                              & \pi \mapsto \set{\set{0} \cup B} \cup (\pi \backslash \set{B}), \quad S \mapsto \set{\set{0} \cup B} \cup (S \backslash \set{B}) \text{ or } S \mapsto S \backslash \set{B},
\end{split}
\]
which produce all possible choices of new such pairs.

\br
Part (b) follows from part (a) since
\[
\ip{\Omega}{X(f_1) \ldots X(f_n) \Omega} = \sum_{\pi \in \NC(n)} \prod_{C \in \Inner(\pi)} \nu[f_C] \prod_{B \in \Outer(\pi)} \mu[f_B].
\]

\br
For part (c), using this expansion for $W(f_0, f_1, \ldots, f_n)$ and applying the recursions gives the following transformations: for $C$ the left-most class of $\pi$,
\[
\begin{split}
& \pi \mapsto \set{0} \cup \pi, \quad S \mapsto S \text{ or } S \mapsto \set{0} \cup S, \\
& \pi \mapsto \set{\set{0} \cup C} \cup (\pi \backslash \set{C}), \quad S \mapsto \set{\set{0} \cup C} \cup (S \backslash \set{C}) \text{ or } S \mapsto S, \\
& \pi \mapsto \set{\set{0} \cup C} \cup (\pi \backslash \set{C}), \quad S \mapsto \set{\set{0} \cup C} \cup (S \backslash \set{C}).
\end{split}
\]
The identical terms in the second and third line come with opposite signs, and as a result one again gets all possible pairs such that $S$ consists of singletons.

\br
For part (d), we show that the sum on the right-hand-side satisfies the recursion of the c-free Appell polynomials. The proof is identical to Proposition 3.22 of \cite{AnsAppell}.
\end{proof}

\subsection{Processes with c-free increments and polynomial martingales}
\label{Subsection:Processes}
There are two natural ways to evolve pairs of states in two-state free probability theory. One way is to choose a freely infinitely divisible state $\rho$ and an arbitrary state $\psi$, and evolve $\phi$ with $\CumFun{\phi, \psi}{\mb{z}} = \CumFun{\rho}{\mb{z}}$ as
\[
\CumFun{\phi(t), \psi}{\mb{z}} = t \CumFun{\rho}{\mb{z}} = \CumFun{\rho^{\boxplus t}}{\mb{z}}.
\]
It is not hard to see that in this case, $\phi(t) = \phi^{\uplus t}$, and so for any $\psi$, we get a Boolean convolution semigroup, which corresponds to a process with Boolean independent increments. The other way is the evolution $\phi(t)$ of Remark~8 of \cite{AnsEvolution}, which corresponds to processes with c-free increments \emph{and} $\psi$-free increments. A Fock space representation in Remark~\ref{Remark:Fock-processes}, or the constructions in \cite{Boz-Spe-Psi-independent}, provide examples of such processes.

\begin{Prop}
\label{Prop:binomial}
Let $\pi \in \Part(n)$ have the property that the collections $\set{X_i: i \in B}_{B \in \pi}$ are c-free with respect to $(\phi, \psi)$ and freely independent with respect to $\psi$. Let $\sigma \in \Int(n)$,
\[
\sigma = (C_1, C_2, \ldots, C_k)
\]
be the largest partition in $\Int(n)$ with $\sigma \leq \pi$. In other words, classes of $\sigma$ are the largest consecutive subsets of classes of $\pi$. Then
\[
\A{\phi, \psi}{X_1, X_2, \ldots, X_n} = \prod_{i=1}^{k-1} \A{\psi}{X_j: j \in C_i} \A{\phi, \psi}{X_j: j \in C_k}.
\]
\end{Prop}

\begin{proof}
It suffices to show that the right-hand side of the equation above satisfies the two conditions in the definition of $A^{\phi, \psi}$. Indeed, for $s \in C_l$, $l < k$,
\[
\begin{split}
& \partial_s \Bigl( \prod_{i=1}^{k-1} \A{\psi}{X_j: j \in C_i} \A{\phi, \psi}{X_j: j \in C_k} \Bigr) \\
&\quad = \prod_{i=1}^{l-1} \A{\psi}{X_j: j \in C_i} \A{\psi}{X_j: j \in C_l, j < s} \otimes \A{\psi}{X_j: j \in C_l, j > s} \A{\phi, \psi}{X_j: j \in C_k} \\
&\quad = \A{\psi}{X_j: j < s} \otimes \A{\psi}{X_j: j \in C_l, j > s} \A{\phi, \psi}{X_j: j \in C_k},
\end{split}
\]
and for $s \in C_k$,
\[
\begin{split}
& \partial_s \Bigl( \prod_{i=1}^{k-1} \A{\psi}{X_j: j \in C_i} \A{\phi, \psi}{X_j: j \in C_k} \Bigr) \\
&\quad = \prod_{i=1}^{k-1} \A{\psi}{X_j: j \in C_i} \A{\psi}{X_j: j \in C_k, j < s} \otimes  \A{\phi, \psi}{X_j: j \in C_k, j > s} \\
&\quad = \A{\psi}{X_j: j < s} \otimes  \A{\phi, \psi}{X_j: j \in C_k, j > s},
\end{split}
\]
where we have used Proposition 3.13 of \cite{AnsAppell}, which states the same factorization property for the free Appell polynomials. Also, by Lemma~\ref{Lemma:Endpoints-c-free},
\begin{multline*}
\state{\prod_{i=1}^{k-1} \A{\psi}{X_j: j \in C_i} \A{\phi, \psi}{X_j: j \in C_k}} \\
= \prod_{i=1}^{k-1} \state{\A{\psi}{X_j: j \in C_i}} \state{\A{\phi, \psi}{X_j: j \in C_k}}
= 0. \qedhere
\end{multline*}
\end{proof}

\noindent
See Remark~5 of \cite{AnsBoolean} or Remark~2.2 of \cite{Boz-Bryc-Two-states} on the relation which the following result bears to the statement ``the c-free Appell polynomials are martingale polynomials''.

\begin{Prop}
If $\set{X_i, Y_i | i = 1, \ldots, n} \subset (\mc{A}^{sa}, \phi, \psi)$, $\mc{B} \subset \mc{A}$ a subalgebra, $\set{X_i} \subset \mc{B} \subset \mc{A}$ and $\set{Y_i}$ are c-free and $\psi$-freely independent from $\mc{B}$, then for any $X \in \mc{B}$,
\[
\state{X \A{\phi, \psi}{X_1 + Y_1, X_2 + Y_2, \ldots, X_n + Y_n}} = \state{X \A{\phi, \psi}{X_1, X_2, \ldots, X_n}}.
\]
\end{Prop}

\begin{proof}
Using the preceding proposition, the expression
\[
\phi \Bigl[ X A^{\phi, \psi} (X_1 + Y_1, X_2 + Y_2, \ldots, X_n + Y_n) \Bigr]
\]
can be expanded as
\[
\begin{split}
\sum_{\substack{\vec{u} = (u(1), u(2), \ldots, u(2k)) \\ u(1) + u(2) + \ldots + u(2k) = n}} & \phi \Bigl[X \A{\psi}{X_1, \ldots, X_{u(1)}} \A{\psi}{Y_{u(1)+1}, \ldots, Y_{u(1) + u(2)}} \\
& \qquad \ldots \A{\phi, \psi}{Y_{u(1) + \ldots + u(2k-1) + 1}, \ldots, Y_{u(1) + \ldots + u(2k)}} \Bigr] \\
+ \sum & \phi \Bigl[X \A{\psi}{X_1, \ldots} \ldots \A{\phi, \psi}{\ldots, X_n} \Bigr] \\
+ \sum & \phi \Bigl[X \A{\psi}{Y_1, \ldots} \ldots \A{\phi, \psi}{\ldots, X_n} \Bigr]
+ \sum \phi \Bigl[X \A{\psi}{Y_1, \ldots} \ldots \A{\phi, \psi}{\ldots, Y_n} \Bigr].
\end{split}
\]
Each term in this sum is of the form $\state{x_1 y_2 x_2 y_2 \ldots}$, with $x_i \in \mc{B}$, $y_i$ c-free and $\psi$-freely independent from $\mc{B}$, and all the terms satisfying $\psi[x_i] = \psi[y_i] = 0$, except $x_1$ and the last term. Applying conditional freeness and Lemma~\ref{Lemma:Endpoints-c-free}, this equals the product
\[
\state{x_1} \state{y_1} \state{x_2} \state{y_2} \ldots.
\]
Moreover, unless there are no $y$ terms at all, the last term is $A^{\phi, \psi}$ and $\phi$ applied to it is zero. So the only non-zero term in the sum is
\[
\state{X \A{\phi, \psi}{X_1, X_2, \ldots, X_n}}. \qedhere
\]
\end{proof}

\begin{Remark}[Fock space representation of processes with c-free increments]
\label{Remark:Fock-processes}
Let $\phi, \psi$ be states on $\mf{C} \langle \mb{x} \rangle$ such that $(\phi, \psi)$ are c-freely infinitely divisible and $\psi$ is freely infinitely divisible. In other words, $R^{\phi, \psi}$ and $R^\psi$ are conditionally positive definite, see \cite{AnsEvolution}. On
\[
\mc{A}_0 = \set{ P(\mb{x}) \in \mf{C} \langle \mb{x} \rangle | P(0) = 0} \otimes L^\infty[0,1],
\]
define functionals
\[
\mu[P(\mb{x}) \otimes f] = \Cum{\phi, \psi}{P(\mb{x})} \int_0^1 f(t) \,dt,
\]
\[
\nu[P(\mb{x}) \otimes f] = \Cum{\psi}{P(\mb{x})} \int_0^1 f(t) \,dt,
\]
Then $\mu, \nu$ are positive semi-definite, so one has a Fock representation for the corresponding c-free Kailath-Segall polynomials as in Proposition~\ref{Prop:Kailath-Segall}. Denote
\[
X_i(t) = X(x_i \otimes \chf{[0,t)}).
\]
Then $\set{X(t)}$ is a process with stationary c-free and $\psi$-freely independent increments.

\br
Let
\[
\state{W(P_1 \otimes f_1, P_2 \otimes f_2, \ldots, P_n \otimes f_n) | t} = W(P_1 \otimes \chf{[0,t)} f_1, P_2 \otimes \chf{[0,t)} f_2, \ldots, P_n \otimes \chf{[0,t)} f_n),
\]
and extend this operation to all of $\mc{A}$. It follows from Proposition~\ref{Prop:Kailath-Segall} that
\[
\state{ \left.A_{\vec{u}}^{X_1(t), \ldots, X_d(t)}(X_1(t), \ldots, X_d(t)) \ \right| s}
= A_{\vec{u}}^{X_1(s), \ldots, X_d(s)}(X_1(s), \ldots, X_d(s)).
\]
\end{Remark}

%\bibliographystyle{amsalpha}
%\bibliography{bibdata}

\def\cprime{$'$}
\providecommand{\bysame}{\leavevmode\hbox to3em{\hrulefill}\thinspace}
\providecommand{\MR}{\relax\ifhmode\unskip\space\fi MR }
% \MRhref is called by the amsart/book/proc definition of \MR.
\providecommand{\MRhref}[2]{%
  \href{http://www.ams.org/mathscinet-getitem?mr=#1}{#2}
}
\providecommand{\href}[2]{#2}

\end{document}